\documentclass[12pt]{amsart}
\usepackage [latin1]{inputenc}
\usepackage{multirow,bigdelim}
\usepackage{amsfonts}
\usepackage{dsfont}
\usepackage{amsmath}
\usepackage{mathrsfs}
\usepackage{todonotes}
\usepackage{hyperref}
\usepackage{multimedia}
\usepackage{graphicx}
\usepackage{indentfirst}
\usepackage{bm}
\usepackage{amsthm}
\usepackage{mathrsfs}
\usepackage{latexsym}
\usepackage{amsmath}
\usepackage{amssymb}
\usepackage{microtype}
\usepackage{relsize}
\usepackage{hyperref}

\def\amslatex{$\mathcal{A}\kern-.1667em\lower.5ex\hbox{$M$}\kern-.125em\mathcal{S}$-\LaTeX}

\DeclareSymbolFont{SY}{U}{psy}{m}{n}
\DeclareMathSymbol{\emptyset}{\mathord}{SY}{'306}

\newtheorem{theorem}{Theorem}[section]
\newtheorem{lemma}[theorem]{Lemma}
\newtheorem{corollary}[theorem]{Corollary}
\newtheorem{proposition}[theorem]{Proposition}

\newtheorem{problem}[theorem]{Problem}

\theoremstyle{definition}
\newtheorem{definition}[theorem]{Definition}

\newtheorem{remark}[theorem]{Remark}

\newcommand{\N}{\mathbb{N}}

\numberwithin{equation}{section}

\begin{document}

\title[On the third  problem of Halmos]
{{\bf On the third  problem of Halmos on Banach spaces}}

\author{ Lixin Cheng$^\dag$, Junsheng Fang$^\ddag$, Chunlan Jiang$^\natural$ }
\address{  Lixin Cheng$^\dag$:  School of Mathematical Sciences, Xiamen University,
 Xiamen 361005, China}
\address{ Junsheng Fang$^\ddag$:  School of Mathematical Sciences, Hebei Normal University, Shijiazhuang, Hebei 050016, China}
\address{ Chunlan Jiang$^\natural$:  School of Mathematical Sciences, Jilin University,
 Changchun, Jilin 130012, China}
 \email{$^\dag$: lxcheng@xmu.edu.cn\;\;(Lixin Cheng)}
 \email{$^\ddag$: jfang@hebtu.edu.cn\;\;(Junsheng Fang)}
 \email{$^\natural$: cljiang@hebtu.edu.cn\;\;(Chunlan Jiang)}

\thanks{$^\dag$ Support in partial
by the Natural Science Foundation of China, grant no. 11731010.\\\indent $^\ddag$ Support in partial
by the Natural Science Foundation of China, grant no.  12071109.\\\indent $^\natural$ Support in partial
by the Natural Science Foundation of China, grant no.  11471271 \& 11471270.}

\date{}

\begin{abstract}
 Assume that $X$ is a complex separable infinite dimensional Banach space and $\mathcal{B}(X)$ denotes the Banach algebra  of all bounded linear operators from  $X$ to itself.
In 1970, P.R. Halmos raised ten open problems in Hilbert spaces. The third one is  the following: If an intransitive operator $T$ has an inverse, is its inverse also intransitive? This question is closely related to the invariant subspace problem. Ever since Enflo's celebrated counterexample on $\ell_1$ answered  the invariant subspace problem in negative, the Banach space setting of the third question of Halmos has become more interesting. In this paper, we give an affirmative answer to this problem under certain spectral conditions. As an application, we show that for an invertible operator $T$ with Dunford's Property ($C$), if $T^{-1}$ is intransitive and there exists a connected component $\Omega$  of ${int}\sigma(T^{-1})^\land$ which is off the origin such that $\Omega\cap\rho_F(T^{-1})\neq \emptyset$, then $T$ is also intransitive. In the end of the paper, we show that a sufficient and necessary condition for that there exists a bounded linear operator without non-trivial invariant subspaces on the infinite dimensional space $L_1(\Omega,\sum,\mu)$ (resp.,  $C(K)$, the space of bounded continuous functions on a complete metric space $K$) is that $(\Omega,\sum,\mu)$ is $\sigma$-finite (resp., $K$ is compact).
\end{abstract}



\maketitle

\section{Introduction}

Assume that $X$ is a complex separable infinite dimensional Banach space and $\mathcal{B}(X)$ denotes the Banach algebra  of all bounded linear operators from  $X$ to itself.
We say that a subspace $Y\subset X$ is an invariant subspace of an operator $T\in\mathcal{B}(X)$ if $TY\subset Y$;
 $T\in \mathcal{B}(X)$ is called intransitive if it leaves invariant spaces other than the trivial space $0$, or, the whole space $X$; otherwise it is transitive.
 The following famous invariant subspace problem was raised a century ago.

\begin{problem}[The invariant subspace problem]
 Does every $T\in \mathcal{B}(X)$ have a nontrivial invariant subspace?
\end{problem}
This problem has been intensively studied, especially for Hilbert spaces, and positive results have been obtained for many classes of operators. (See, for instance \cite{RR2,LS2}.)

The celebrated counterexample of Enflo \cite{En} gave a negative answer to the problem. It was done by showing that  there is an operator
 $T\in\mathcal B(\ell_1)$ which does not admit a nontrivial invariant subspace of $\ell_1$.
 The paper has existed in manuscript form for about 12 years prior to its publication, the delay being due to its enormous complexity.
 In the meantime, constructions of operators without invariant subspaces have been published by
 C. Read \cite{R1,R2,R3} 
and B. Beauzamy \cite{Bea}. Read \cite{R2,R3} showed that every  infinite dimensional separable Banach space $X$ admits a bounded linear operator without nontrivial invariant subspaces whenever $X$ contains $\ell_1$, or, $c_0$ as its complemented subspaces.  On the other hand, Beauzamy's example possesses a cyclicity property stronger than the nonexistence of invariant subspaces. The construction of an operator without a nontrivial invariant subspace on a nuclear Fr\'{e}chet space (a rather different problem) has been achieved by A. Atzmon \cite{at}.  In 2011,
S.A. Argyros and R.G. Haydon \cite{Ar} constructed a remarkable Banach space $X_K$ that solves the scalar-plus-compact problem.
  The space $X_K$ is the first example of a Banach space for which it is known that every bounded linear operator on the space has the form $\lambda+K$
  where $\lambda$ is a real scalar and $K$ is a compact operator, it is also the first example of an infinite-dimensional space on which
  every bounded linear operator has a nontrivial invariant subspace.

   However, there are many important questions left unknown.
   For example, we do not know that whether every bounded linear operator on a separable Hilbert space admits a nontrivial invariant subspace; we do not even know  if there exists a bounded linear operator on a reflexive Banach space that does not have a non-trivial invariant subspace. Therefore, mathematicians have not stopped the pace of further research on these and related questions.      

A problem closely related to the invariant subspace problem is the third problem of Halmos \cite{Hal}.
\begin{problem}[The third problem of Halmos]\label{1.1} If an intransitive operator has an inverse, is its inverse also intransitive?
\end{problem}
This problem has been around for over fifty years. Not only is the answer to this problem still unknown, but little research has progressed. 16 years ago,  R.G. Douglas, C. Foias and C. Pearcy  hoped their joint paper \cite{DFP} will stimulate
interest in this problem of Halmos -- the only problem of the ten set forth in \cite{Hal}
which presently remains open.

In 1970, P.R. Halmos \cite{Hal} raised ten open problems in a separable Hilbert space.  The first problem asks whether the set of cyclic operators has a nonempty interior and has been subsequently answered in negative by  P.A. Fillmore, J.G. Stampfli and J.P. Williams \cite{FSW}. The second problem asks if every part of a weighted shift is similar to a weighted shift. Halmos himself \cite{Hal} presents a characterization due to C. Berger (unpublished) under the assumption that a weighted shift is subnormal.
   Whether the inverse of an operator that is known to have an invariant subspace must also have one is the third one (Problem 1.2). In 1975, D. Hadwin \cite[Theorem C]{Ha2} showed that for every separable complex Banach space there is an unbounded linear bijection $S$ on it with a nontrivial invariant subspace so that its algebraic inverse $S^{-1}$ fails to have a  nontrivial invariant subspace.    
   The fourth problem asks whether every normal operator on a separable space is the sum of a diagonal operator and a compact one. The first proof of this is due to I.D. Berg (unpublished), while an elegant one is  due to Halmos himself. 
    The fifth problem asks if every subnormal Toeplitz operator is either analytic or normal, which was  resolved independently by  C.C. Cowen \cite{CC}, J. Long and S. Sun \cite{S1}, \cite{S2}.
      The sixth problem asks  if every polynomially bounded operator is similar to  a contraction? This problem was answered in negative by  G. Pisier \cite{Pis} in 1997. 
    Problem seven asks if  every  quasinilpotent  operator is the norm  limit  of  nilpotent  ones.  It was  solved by  C. Apostol and  D. Voiculescu \cite{Voi} in 1974.  Problem eight asks  if every operator is the norm  limit  of  reducible  ones. It was solved by D. Voiculescu \cite{Voi2} in 1976.   Problem nine asks  whether every  complete  Boolean is algebra  reflexive.   Problem ten asks  whether   every  non-trivial  strongly  closed  transitive  atomic  lattice is either  medial or self-conjugate.

Recall that a nonzero vector $x$ in $X$ is called a cyclic vector of $T$ provided $[T^kx]\equiv[T^kx]_{n=0}^\infty=X$,
where $[T^kx]_{n=0}^\infty$ denotes the norm closure of ${\rm span}\{T^kx\}_{n=0}^\infty$ in $X$. We denote by $C(T)$ the set of all cyclic vectors of $T$. 

 Problem \ref{1.1}  can be expressed in the following way: If an invertible operator $T$ has no nontrivial invariant subspaces, so does its inverse $T^{-1}$?
Or, does the assumption that $C(T)=X\setminus\{0\}$ imply $C(T^{-1})=X\setminus\{0\}?$  Therefore,
 the answer to Problem \ref{1.1} would be positive if we could show $C(T)= C(T^{-1})$ for every invertible operator $T$. Unfortunately, R.G. Douglas, C. Foias and
 C. Pearcy \cite{DFP} have shown that there are three classes of bounded linear operators in a Hilbert space so that for each $T$ in these classes
  we have $C(T)\neq C(T^{-1})$.\\

All symbols and notations presented in this paper are standard.  For an operator $T\in \mathfrak B(X)$,  $\sigma(T)$ denotes the spectrum of $T$ and $\rho(T)=\mathbb C\setminus\sigma(T)$; We use ${\rm Lat}T$ to denote the invariant subspace lattice of $T$.
For a compact subset $K$ in $\mathbb C$, the polynomial convex hull $K^\land$ of $K$
was introduced by E.L. Stout \cite[p.23]{St}.
\begin{equation}\label{1.2}
K^\land=\left\{z\in\mathbb C:\, |f(z)|\leq \sup_{w\in K}|f(w)|\,\text{for all polynomials }f \right\}.
\end{equation}
Thus, the definition of $\sigma(T)^{\land}$ is self-explanatory.
The following property is due to E.L. Stout \cite[p.24]{St}.
\begin{proposition}\label{1.3}
Let $X$ be a complex Banach space and $T\in\mathfrak B(X)$. Then $\sigma(T)^{\land}$ is the union of $\sigma(T)$ and all of the bounded connected components of $\mathbb C\setminus\sigma(T)$.
\end{proposition}
If $T\in\mathfrak B(X)$ is  invertible, then for a fixed $x_0\in X$, we put
\begin{equation}\label{1.4}\mathfrak N=[T^{-k}x_0:\, k\geq 1].\end{equation}
Keeping these notations in mind, we state the first main result of paper as follows.
\begin{theorem} \label{1.5}
 Suppose that $T\in\mathfrak B(X)$ is an invertible operator, $x_0\in X$ is a nonzero noncyclic vector of $T^{-1}$, and $U=T^{-1}|_{\mathfrak N}$. If $\sigma(U)^\land\bigcap\rho_{\rm F}(U)$ has a connected component which does not contain the origin, then $T$ is intransitive, where $\rho_F(U)=\{\lambda:\, \lambda\in \mathbb{C}\,\text{and}\, U-\lambda\,\text{is a Fredholm operator}\}$.
\end{theorem}

Note that $\mathfrak N\in{\rm Lat}T^{-1}.$ Then $T^{-1}$ induces a quotient operator $(T^{-1})^\mathfrak N: X/\mathfrak N\rightarrow X/\mathfrak{N}$ as follows
\begin{equation}\label{1.6} (T^{-1})^\mathfrak N(x+\mathfrak N)=T^{-1}x+\mathfrak N,\; \;x\in X.\end{equation}
Now, we are ready to state our second main result below.
\begin{theorem}\label{1.7}
Let $T\in \mathfrak B(X)$ be invertible and ${\rm int}\sigma(T^{-1})^\land$ contain  two components $\Omega_0$,  $\Omega_1$ with $0\in\Omega_0$  and $\Omega_1\bigcap\rho(T^{-1})\neq\emptyset$. If $T^{-1}$ admits a nonzero noncyclic vector  $x_0\in X$ such that $\sigma(T^{-1})\subset\sigma((T^{-1})^\mathfrak N)$, then $T$ is intransitive.
\end{theorem}
 For the study related  the  single  valued  extension of the  resolvent $(T-\lambda)^{-1}$ of an operator $T\in\mathfrak B(X)$, N. Dunford \cite{Dun} introduced three important properties (A), (B) and (C).

The family of operators admitting Property ($C$) is a large class of $\mathcal{B}(X)$. For example, it contains the classes of  generalized spectral operators and of decomposable operators ; and of hypo-normal operators (an operator $T$ is called hypo-normal if $T^*T-TT^*\geq 0$), and of $N$-class operators (an operator $T$ is called $N$-class operator if $\|T\xi\|^2\leq \|T^{2N}\xi\|\cdot \|\xi\|$ for all $\xi \in \mathcal{H}$) \cite{li}(Theorem 2) in a separable Hilbert space.

As their applications of Theorems \ref{1.5} and \ref{1.7}, we obtain the following results.
\begin{theorem}\label{7.10}
 Suppose that $T\in \mathfrak B(X)$ is  invertible, and  that $T^{-1}$ is intransitive with Dunford's Property ($C$).  If $\sigma(T^{-1})^\land\bigcap \rho_F(T^{-1})$ has a connected component off the origin,  then $T$ is intransitive.
\end{theorem}
\begin{theorem}\label{1.8}
Let Let $X$ be a complex Banach space and $T\in \mathcal{B}(X)$ be an invertible operator.  If $T^{-1}$ has a proper strictly cyclic invariant subspace and there exists a nonempty bounded open connected component $\Omega_1$  of $\rho(T^{-1})$ such that $\Omega_1\bigcap\Omega_0=\emptyset$, where $\Omega_0$ is the connected component of ${\rm int}(\sigma(T^{-1})^\land)$ containing the origin. Then $T$ is intransitive.
\end{theorem}

With the help of some interesting results of Read \cite{R4, R5}, finally, we show the following theorems.
\begin{theorem}\label{1.9}
A sufficient and necessary condition for an infinite dimensional $L_1(\Omega,\sum,\mu)$ admitting an operator $T\in\mathfrak B(L_1(\Omega,\sum,\mu))$ without any nontrivial subspace is that the measure space
$(\Omega,\sum,\mu)$ is $\sigma$-finite.
\end{theorem}

\begin{theorem}\label{1.10}  Let $K$ be a complete metric space such that the space $C(K)$ of all bounded continuous complex valued functions is infinite dimensional. Then
 a sufficient and necessary condition for $C(K)$ admitting an operator $T\in\mathfrak B(C(K))$ without any nontrivial subspace is that the metric space $K$
 is compact.
\end{theorem}



\begin{remark}
L. Cheng, J. Fang and C. Jiang contributed equally to this work. The authors would like to thank Professor Zhang Yuanhang for his helpful discussions on the paper.
\end{remark}

\section{Preliminaries}
In this paper, all notations and symbols are standard. Unless stated explicitly otherwise, we will assume that $X$ is a complex separable infinite dimensional Banach space and $X^*$ its dual. We denote by $\mathcal{B}(X)$  the Banach algebra of all bounded linear self-mappings defined on $X$, and by  ${\mathcal K}(X)$ be the closed  ideal of all compact operators in $\mathcal{B}(X)$, i.e., all operators in $\mathcal{B}(X)$  which maps every bounded subset of $X$ into a relatively compact subset. Therefore, the quotient space $\mathcal{B}(X)/{\mathcal K}(X)$ is the Calkin algebra, that is, the natural quotient mapping  $\pi: \mathcal{B}(X)\rightarrow  \mathcal{B}(X)/{\mathcal K}(X)$ is  a homomorphism projection.

For an operator $T\in\mathcal{B}(X)$, ${\rm ker} T$ (resp.,${\rm ran} T$) denotes the kernel (resp., the range) of $T$. If $M\subset X$ is an invariant subspace of $T$, then $T|_M$ denotes the restriction of $T$ to $M$; $\sigma(T)$ denotes the spectrum of $T$. i.e., $$\sigma(T)=\{\lambda\in\mathbb C: T-\lambda\;{\rm is \;not\;invertible}\},$$  and its complement $\mathbb C\setminus \sigma(T)$ is denoted by $\rho(T)$;  $\sigma_{p}(T)$ stands for the point spectrum of $T$, that is, $\sigma_{p}(T)=\{\lambda: \mbox{dim}\ker (T-\lambda)\geq1\}$;
$\sigma_{e}(T)$ stands for the  essential spectrum of $T$, that is, the spectrum of $\pi(T)$ in $ \mathcal{B}(X)/{\mathcal K}(X)$.

 An operator $T\in\mathcal{B}(X)$ is said to be Fredholm (resp., semi-Fredholm) provided $\mbox{ran}(T)$ (the range of $T$) is closed, and both of the dimension of its kernel ker$T$ and the codimension of ran$T$ are finite dimensional, i.e., $dim(X/ranT)<\infty$ (resp., either  ker$T$, or codimension of ran$T$ is finite dimensional).  $\rho_{F}(T)\equiv\mathbb{C}\backslash\sigma_{e}(T)$ is said to be the Fredholm domain of $T$. The well-known  Atkinson theorem states that $T\in\mathcal{B}(X)$ is Fredholm if and only if $\pi(T)$ is invertible in $\mathcal{B}(X)/{\mathcal K}(X)$.
${\rm Lat}T$ represents the invariant subspace lattice of $T$.

  For a subset $A\subset X$, $[A]$ is  the closure of the linear hull ${\rm span}(A)$ of $A$. If $A=\{x_n\}_{n=k}^\infty\subset X$ is a sequence,
  then we denote the closure of the linear hull ${\rm span}(A)$ by $[x_n]_{n\geq k};$  if no confusion arises, we simply write it as $[x_n]$.
   $A^\bot\equiv\{x^*\in X^*: \langle x^*,x\rangle=0,\;\forall\;x\in A\}$ is called the annihilator of $A$.
   If $A\subset X^*$, then we also use $[A]^{w^*}$ to denote the weak-star closure of the linear hull ${\rm span}(A)$ in $X^*$.

Recall that an operator $P\in \mathcal{B}(X)$ is said to be an idempotent if $P^{2}=P$; and an nontrivial idempotent $P$ means the idempotent $P$ is neither  $0$, nor the identity $I$. For $T\in \mathcal{B}(X)$, its commutant $\mathcal{A}'(T)$ is defined by $$\mathcal{A}'(T)=\{S\in \mathcal{B}(X): ST=TS\}.$$

The following definition of strongly irreducible operators was introduced independently by Gilfeather~\cite{Gil} and Jiang Zejian~\cite{Jia} in 1970s.

\begin{definition}\label{2.1}
An operator $T\in \mathcal{B}(X)$ is called strongly irreducible provided the commutant of $T$ does not admits any nontrivial idempotent operator. Otherwise, it is called strongly reducible.
\end{definition}

\begin{definition}\label{2.2} Let $T\in \mathcal{B}(X)$.
i) A (nonzero) element $x\in X$ is said to be a cyclic vector of  $T$ provided $[T^{k}x,\,k\geq0]$ (the closure of the linear hull span$\{T^{k}x,\,k\geq0\}$) is the whole space $X$. We denote by $\mathcal{C}(T)$ the set of all cyclic vectors of $T$.

ii) We say that $T$ is transitive provided $\mathcal{C}(T)=X\backslash\{0\}$.
\end{definition}

Clearly, if an operator $T\in \mathcal{B}(X)$ is strongly reducible, then $T$ is intransitive. 
\begin{proposition}\label{2.3}
Let $T\in \mathcal{B}(X)$ be an invertible operator. If $T$ is strongly reducible, then

i) $T^{-1}$ is also strongly reducible;

ii)  $T$ and $T^{-1}$ have a common nontrivial invariant subspace.
\end{proposition}
\begin{proof}
Suppose that $P\in \mathcal{A}'(T)$ is a nontrivial idempotent operator. Then it follows from  definition of the commutant $\mathcal{A}'(T)$ and invertibility of $T$ that  $P\in \mathcal{A}'(T^{-1})$. Therefore, both ${\rm ran}P$ and ${\rm ran}(I-P)$ are common nontrivial invariant subspaces of $T$ and $T^{-1}$.
\end{proof}

\begin{corollary}\label{2.4}
Let $T\in \mathcal{B}(X)$ be an invertible operator. If its spectrum $\sigma(T)$ is not connected, then $T$ and $T^{-1}$ have a common nontrivial invariant subspace.
\end{corollary}
\begin{proof}
Suppose that $\sigma(T)$ is not connected.
Then by  Riesz' functional calculus,  there is a nontrivial idempotent $P$ such that $P\in \mathcal{A}'(T)\bigcap \mathcal{A}'(T^{-1})$.
Therefore,  ${\rm ran}P\in {\rm Lat}T\bigcap {\rm Lat}T^{-1}$.
\end{proof}

\begin{proposition}\label{2.5}
Let $T\in \mathcal{B}(X)$ be invertible. If $\sigma_{p}(T)\cup\sigma_{p}(T^{*})\neq\emptyset$, then $T$ and $T^{-1}$ have  a common nontrivial invariant subspace.
\end{proposition}
\begin{proof}
Without loss of generality, we assume that $T$ is not a scalar operator.
Suppose that $\lambda\in \sigma_p(T)\bigcup \sigma_p(T^*)$.  The  two cases below may occur.

(1) If $\lambda\in \sigma_p(T)$, then there is $0\neq x \in X$ such that $(T-\lambda)x=0$, that is,
$T^{-1}((T-\lambda)x)=(I-\lambda T^{-1})x=0$. Therefore, $x\in \ker(T^{-1}-\lambda^{-1})$, and $\ker(T-\lambda)\in {\rm Lat}(T^{-1})\bigcap{\rm Lat}(T)$.

(2) If $\lambda\in \sigma_p(T^*)$, then for $x\in \ker(T^{*}-\lambda)$, we obtain $x\in \ker((T^{*})^{-1}-\lambda^{-1})$. Thus, $$\ker(T^*-\lambda)^{\bot}\in {\rm Lat}(T^{-1})\bigcap{\rm Lat}(T).$$
\end{proof}

Recall that $\sigma(T)^{\land}$ denotes the polynomial convex hull of  $\sigma(T)$. Note that $\sigma(T)^{\land}$ is the union of $\sigma(T)$ and all of the bounded connected components of $\mathbb C\setminus \sigma(T)$ (Proposition \ref{1.3}).
\begin{proposition}\label{2.6}\cite{DFP}
Let $T\in \mathfrak{B}(X)$ be an invertible operator. If $0\notin \sigma(T)^{\land}$, then $${\rm Lat}(T)={\rm Lat}(T^{-1})\;\; \mbox{and}\;\; C(T)=C(T^{-1}).$$
\end{proposition}

In the following sections, we always assume that $T\in \mathfrak{B}(X)$ is invertible and $0\in \sigma(T)^\land$.


We conclude  this section by the following well-known result.

\begin{lemma}\label{2.8}
Let $S\in \mathfrak B(X)$ and $E$ be an invariant subspace of $S$.
Then $\sigma(S|_{E})\subseteq\sigma(S)^\land$ and $\rho(S)\subseteq\rho_F(S|_E)$.
\end{lemma}
\section{Existence of right inverses}
In this section, we show a characterization related to existence of the right inverse of a bounded linear operator.

\begin{theorem}\label{3.1}
 $T\in\mathfrak B(X)$ admits a right inverse, i.e., there is $S\in\mathfrak B(X)$ so that $TS=I$ if and only if $T$ is surjective and ${\rm ker}T$ is complemented in $X$.
 Consequently,  $TS=I$ implies that $X={\rm ker}T\oplus{\rm ran}S$.
\end{theorem}
\begin{proof}
Sufficiency. Since $TX=X$, $X/{\rm ker}T$ is isomorphic to $X$. Let $Y\subset X$ be a closed subspace so that $X=Y\oplus{\rm ker}T$. Then $Y\approx \,\text{(isomorphic)}\, X$. Let $P: X\rightarrow Y$ be the projection from $X$ along ${\rm ker}T$ to $Y$. Then the restriction $U\equiv TP|_Y=T|_Y$ is an isomorphism from $Y$ to $X$. Therefore,
$U^{-1}: X\rightarrow Y$ is again an isomorphism. Consequently, $TU^{-1}: X\rightarrow X$ is a self-isomorphism of $X$. We finish the proof of the sufficiency by letting $S=U^{-1}(TU^{-1})^{-1}$.

Necessity. Suppose that there is $S\in\mathfrak B(X)$ so that $TS=I$. Clearly, $T$ is surjective and $S^*T^*=(TS)^*=Id_{X^*}$. Thus, $S^*:X^*\rightarrow X^*$ is surjective,  which further entails that $S: X\rightarrow X$ is injective. We write \[Y={\rm ran}S,\;{\rm and\;}\;Z={\rm ker}T.\] Note that $TS=I$ implies that  $T|_Y: Y\rightarrow X$ 
is an  isomorphism.  Therefore, $Y$ is a closed subspace of $X$. This and injectivity of $S$ entail that $S: X\rightarrow Y$ is again an isomorphism.
Next, we define two mappings as follows.
\[T_{X/Z}: X/Z\rightarrow X\;\;{\rm by\;\;}T_{X/Z}(\bar{x})=Tx,\;\bar{x}=x+Z\in X/Z,\;\;\;\;\;\;\;\;\;\;\;\;\;\;\;\;\]
and
\[T_{Y/Z}: Y/Z\rightarrow X\;\;{\rm by\;\;}T_{Y/Z}(\bar{Sx})=(TS)x=x,\;\bar{Sx}=Sx+Z\in Y/Z.\]
Clearly, both $T_{X/Z}$ and $T_{Y/Z}$ are isomorphisms. Since $Y/Z\subset X/Z$ and since the restriction $(T_{X/Z})|_{Y/Z}$ of $T_{X/Z}$ to $Y/Z$ is just $T_{Y/Z}$, we obtain that $$Y/Z=X/Z.$$ Therefore, to show $Z$ is complemented in $X$, it suffices to prove
$Y+Z=Y\oplus Z$, i.e., $Y$ and $Z$ are complemented each other in $Y+Z$. Suppose, to the contrary, that there are two sequences $\{y_n\}\subset Y$,  $\{z_n\}\subset Z$, $\|y_n\|=\|z_n\|=1$, such that $\|y_n-z_n\|\rightarrow 0$. Then \[0=\lim_n\|T\|\cdot\|y_n-z_n\|\geq\lim_n\|T(y_n-z_n)\|=\lim_n\|T(y_n)\|=\lim_n\|T|_Y({y}_n)\|.\]
This contradicts to that $T|_Y: Y\rightarrow X$ is an isomorphism.
\end{proof}
\section{Cowen-Douglas operators}
In this section, we introduce some  properties of Cowen-Douglas operators of index 1 defined on $X$, which play an important role in the proof of the main results of this paper.

Assume that $\Omega$ is a nonempty bounded open connected subset of the complex plane
$\mathbb{C}$, and  $n\in\N$.  In 1978, M.J. Cowen and R.G. Douglas \cite{CD}
introduced a class of operators denoted by $B_n(\Omega)$
which contains $\Omega$ as eigenvalues  of
constant multiplicity $n$. Recall that $$\sigma(T)=\{\lambda\in \mathbb{C}:T-\lambda ~~
\mbox{is not invertible}\},$$ and the closure\;of\;span$\{\mbox{ker}(T-\lambda):\,\lambda\in \Omega\}$ is denoted by $$[\mbox{ker}(T-\lambda):\,\lambda\in \Omega].$$ The class of Cowen-Douglas operators of
rank $n$ on $X$ -- denoted by $B_n(\Omega)$ is defined as follows:
$$\begin{array}{lll}B_n(\Omega)=\Big\{T\in \mathcal{B}(X):
&(1)\,\,\Omega\subset \sigma(T);\\
&(2)\,\,[\mbox{ker}(T-\lambda):\,\lambda\in \Omega]=X;\\
&(3)\,\,\mbox{ran}(T-\lambda)=X,\;{\rm for\;each}\;\lambda\in\Omega;\;{\rm and}\\
&(4)\,\,\mbox{dim ker}(T-\lambda)=n,\; {\rm for\;each}\;\lambda\in\Omega.\Big\}
\end{array}$$

By Theorem \ref{3.1}, the result below follows.
\begin{corollary}\label{3.2}
Assume that $\Omega$ is a nonempty bounded open connected subset of
$\mathbb{C}$. Then for all $n\in\N$, every $T\in B_n(\Omega)$ has a right inverse.
\end{corollary}

\begin{lemma}\label{3.3}
Suppose $\Omega$ is a nonempty bounded open connected subset of
$\mathbb{C}$,  $T\in B_1(\Omega)$ and that $\lambda_0\in \Omega$. Then
\begin{equation}\label{3.4}
[{\rm ker}(T-\lambda_0)^{k},\,k\geq1]=X=[{\rm ker}(T-\lambda),\,\lambda\in\Omega].\end{equation}
\end{lemma}
\begin{proof}
Without loss of generality, we can assume that $\lambda_0=0$; Otherwise, we substitute $T-\lambda_0$ for $T$.
 Since $T$ is surjective with ${\rm dim ker}T=1$, by Corollary \ref{3.2}, there exists an operator $S\in \mathfrak B(X)$ such that $TS=I$.
  Given $ x_0\in {\rm ker}T$, choose $0<\delta<\|S\|$ and let $O_\delta=\{\lambda\in\mathbb C: |\lambda|<\delta\}$. Then 
 \[S_\lambda\equiv\sum_{n=0}^\infty \lambda^nS^n\in \mathfrak B(X),\] and $S_{(\cdot)}$ is a holomorphic map from $O_\delta$ into $\mathfrak B(X)$. Note that \[TS_\lambda(x_0)=\lambda S_\lambda(x_0), {\rm for \;all\;}\lambda\in O_\delta.\] Then \[S_\lambda (x_0)\in{\rm ker}(T-\lambda),\;{\rm for \;all\;}\lambda\in O_\delta.\]

We first show the second equality in (\ref{3.4}), i.e., $[{\rm ker}(T-\lambda),\,\lambda\in\Omega]=X$, which is equivalent to
\begin{equation}\label{3.5}[S_\lambda(x_0):\,\lambda\in O_\delta]=X.\end{equation}
Notice that $S_\lambda(x_0)$ is holomorphic on $O_\delta$ and that $T\in B_1(\Omega)$. There
 exists a holomorphic frame $e(\cdot):\,\Omega\rightarrow X$ such that $e(\lambda)=S_\lambda(x_0)$  and
  $Te(\lambda)=\lambda e(\lambda)$ for all $\lambda\in O_\delta$.
To prove \ref{3.5}, it suffices to show that for any $x^*\in X^*$,
\begin{equation}\label{3.6}\langle x^*, S_\lambda(x_0)\rangle=0\;{\rm for\; all\;} \lambda\in O_\delta\;\;{\rm implies\;}x^*=0.\end{equation}
Assume that $\langle x^*, S_\lambda(x_0)\rangle=0\;{\rm for\; all\;} \lambda\in O_\delta$. Then $\langle x^*, e(\lambda)\rangle=0\;{\rm for\; all\;} \lambda\in O_\delta$.
Note that $\langle x^*, e(\cdot)\rangle$ is a holomorphic function on $\Omega$. We obtain that  $\langle x^*, e(\lambda)\rangle=0,\;{\rm for\; all\;} \lambda\in\Omega,$ and which entails
 $x^*=0$. Thus, (\ref{3.5}) holds. The second equality in (\ref{3.4}) has been shown.


To show the first equality in (\ref{3.4}),   note that \[\langle x^*,S_\lambda (x_0)\rangle=\sum_{n=0}^\infty \lambda^n\langle x^*,S^n(x_0)\rangle\] is a holomorphic function of $\lambda\in O_\delta$ for every $x^*\in X^*$. Then a sufficient and necessary condition for $\langle x^*,S_{(\cdot)} (x_0)\rangle=0$ on $O_\delta$ is that $\langle x^*,S^n (x_0)\rangle=0$  for all $n\geq 0$. Therefore, (\ref{3.5})  entails $[S^n(x_0):\, n\geq 0]=X$. Since $${\rm ker} T^k=[x_0,S(x_0),\cdots,S^{k-1}(x_0)]$$ for all $k\in\N$, we obtain
$[{\rm ker} T^k:k\in\N]=X$. This is just the first equality in (\ref{3.4}).
\end{proof}
The following property immediately follows from Lemma~\ref{3.3}.

\begin{lemma}\label{3.7}
Suppose that $\Omega_{1},\ \Omega_{2}$ are bounded connected open subsets of $\mathbb{C}$. If $\Omega_{1}\subset\Omega_{2}$, then $B_{1}(\Omega_{2})\subseteq B_{1}(\Omega_{1})$.
\end{lemma}

\begin{lemma}\label{3.8}
Suppose that $\Omega$ is a nonempty bounded open connected subset of
$\mathbb{C}$, and $T\in B_1(\Omega)$. Then

i) ${\rm ran}(T-\lambda)$ is dense in $X$ for all $\lambda\in \mathbb{C}$;

ii) $\sigma_{p}(T^*)=\emptyset$:

iii) $T$ is strongly irreducible; consequently,

iv)  $\sigma(T)$ is connected.
\end{lemma}
\begin{proof}
 Assume that $e(\lambda)$ is a non-zero eigenvector of $T$ with respect to $\lambda\in\Omega$, i.e., $Te(\lambda)=\lambda e(\lambda)$. Then by Lemma \ref{3.3}, $$[e(\lambda):\,\lambda\in\Omega]=[\mbox{ker}(T-\lambda):\,\lambda\in\Omega]=X.$$  Note that  \[(T-\gamma)e(\lambda)=(\lambda-\gamma)e(\lambda),\,{\rm for\;all\;}\lambda\in\Omega,\gamma\in \mathbb{C},\]
 and that  ${\rm ran}(T-\lambda)=X$, whenever  $\lambda\in \Omega$. 

For any fixed $\gamma\in\mathbb C\setminus\Omega$, \[[(T-\gamma)e(\lambda):\,\lambda\in \Omega]=[(\lambda-\gamma)e(\lambda):\,\lambda\in \Omega]=[e(\lambda):\,\lambda\in \Omega]=X.\]
This shows that ${\rm ran}(T-\gamma)$ is dense in $X$, consequently, $\ker(T-\gamma)^*=\{0\}$. Therefore, ii) $\sigma_{p}(T^*)=\emptyset$ has been shown.

To show iii), we first note that for every $\lambda\in\Omega$, $\ker(T-\lambda)=[e(\lambda)]$. Suppose, to the contrary, that there exists a nontrivial idempotent $P\in \mathcal{A}'(T)$. Then $$(TP)e(\lambda)=(PT)e(\lambda)=\lambda Pe(\lambda).$$
 Thus, $Pe(\lambda)\in \ker(T-\lambda).$  Analogously, we obtain $$(I-P)e(\lambda)\in \ker(T-\lambda).$$
 Since $[e(\lambda):\lambda\in \Omega]=X$, both $Pe(\lambda)$ and $(I-P)e(\lambda)$ are  non-zero holomorphic  valued functions.
Since the two vectors $Pe(\lambda)$ and $(I-P)e(\lambda)$ are linearly independent,  dimker$(T-\lambda)\geq 2$. This contradicts to  $T\in B_1(\Omega).$
While iv) is a direct consequence of iii).
\end{proof}

\begin{lemma}\label{3.9}
For every $T\in B_{1}(\Omega)$, there exists a connected open subset $\Omega_{\rm max}$ of $\mathbb C$ such that
\begin{enumerate}
\item  \;$\Omega\subseteq\Omega_{\rm max}$;
\item \; $B_{1}(\Omega_{\rm max})\subseteq B_{1}(\Omega)$;
\item \; $\sigma_{e}(T)\supset\partial\Omega_{\rm max}$, the boundary of $\Omega_{\rm max}$.
\end{enumerate}
In this case, we call $\Omega_{\rm max}$ the maximal domain of $T$.
\end{lemma}
\begin{proof}
(1) \;It follows from  Zorn's Lemma, there is a maximal bounded connected open subset of $\mathbb{C}$ denoted by $\Omega_{\rm max}$ such that
 $T\in B_{1}(\Omega_{\rm max})$. Due to Lemma \ref{3.7}, (2) follows.

Suppose, to the contrary of (3), that  there exists $\lambda_0\in\partial\Omega_{\rm max}$ such that $\lambda_0$ in $\rho_{F}(T)$. Therefore, there exists a neighborhood $O_{\lambda}$ of $\lambda_0$ such that $$\text{ind}\,(T-\lambda)=1,\,{\rm for\;all}\,\lambda\in O_{\lambda_0}.$$
By Lemma~\ref{3.4}, $\sigma_{p}(T^{*})=\emptyset$. Note dim$\,\ker(T-\lambda)=1$.
Let $\Omega_1=\Omega_{\rm max}\bigcup O_{\lambda}$. Then $\Omega_1$ is again open connected. Thus, $T\in B_{1}(\Omega_1)$. This contradicts to the maximality of $\Omega_{\rm max}$.
\end{proof}



\begin{lemma}\label{3.10}
Suppose that  $\Omega\subset \mathbb C$ is a connected open subset with $0\in\Omega$. Let $S,T\in\mathfrak B(X)$ with $T\in B_1(\Omega)$  and  $TS=I$. If  $0\neq x_{0}\in \ker T$, then $x_{0}\in C(S)$. In particular, $C(S)\neq\emptyset$.
\end{lemma}
\begin{proof} Note $\ker T=[x_0]$.
 $TS=I$ implies $\ker T^k=[x_{0},Sx_{0},\cdots, S^{(k-1)}x_{0}].$
It follows that
$[S^{k}e_{0},k\geq0]=X.$ Therefore,  $x_{0}\in \mathcal{C}(S)$.
\end{proof}

The following result is due to P.A. Fillmore, J.G. Stampfli and J.P. Williams~\cite{FSW}. (See, also ~\cite{Her}.)
\begin{proposition}\label{3.11}{\rm \cite{FSW}}
Let $T\in \mathcal{B}(X)$. If there is  $\lambda\in\mathbb{C}$ such that $${\rm dimker}(T-\lambda)^{*}\geq2,$$ then $C(T)=\emptyset$.
\end{proposition}

For $T\in\mathfrak B(X)$, we use $\rho_{\text{s-F}}(T)$ to denote the semi-Fredholm domain of $T$, that is, a sufficient and necessity condition for $\lambda\in\rho_{\text{s-F}}(T)$ is that $\mbox{ran}(T-\lambda)$ (the range of $T-\lambda$) is closed, and  either  ker($T-\lambda$), or ran($T-\lambda$) is finite dimensional.  This is equivalent to that $T-\lambda$ is semi-Fredholm.
The following result is due to D.A. Herrero \cite{Her}.
\begin{proposition}\label{3.12}{\rm \cite{Her}}
Let $T\in \mathcal{B}(X)$. If $\mathcal{C}(T)\neq\emptyset$, then \[\rho_{\text{s-F}}^{-1}(T)\equiv\Big\{\lambda:\, \text{ind}(T-\lambda)=-1\Big\}\] is simply connected.
\end{proposition}

\begin{lemma}\label{3.13}
Let $A,B\in \mathcal{B}(X)$ such that $AB=I$. If ${\rm dimker} A=1$ and $0\neq x\in {\rm ker} A$ such that $[B^nx]_{\{n\geq 0\}}=X$. Then $A\in B_1(\Omega_A)$, where $\Omega_A$ is the connected component of $\rho_{\text{s-F}}(A)$ containing the origin.
\end{lemma}
\begin{proof}
Choose $0<\delta<\|B\|$ and let $O_\delta=\{\lambda\in \mathbb{C}:\, |\lambda|<\delta\}$. For $\lambda\in O_\delta$,
\[
(A-\lambda)B=AB-\lambda B=I-\lambda B.
\]
Note that $I-\lambda B$ is invertible for $\lambda\in O_\delta$. Thus, ${\rm ran}(A-\lambda)=X$.

Since $\text{ind} A=1$, $\text{ind} (A-\lambda)=1$ for $\lambda \in O_\delta$ by the continuity of index. In particular, $\text{dimker}(A-\lambda)=1$ for $\lambda\in O_\delta$. For $\lambda\in O_\delta$, let
\[
B_\lambda=\sum_{n=0}^\infty\lambda^nB^n\in \mathcal{B}(X).
\]
Then $B_\lambda$ is a holomorphic map from $O_\delta$ into $\mathcal{B}(X)$. Note that
\[
AB_\lambda(x)=\lambda B_\lambda(x), \,\text{for all}\, \lambda\in O_\delta.
\]
Then $B_\lambda(x)\in ker (A-\lambda)$ for all $\lambda\in O_\delta$. Note that
\[
[B_\lambda(x)]_{\lambda\in O_\delta}=[B^nx]_{\{n\geq 0\}}=X.
\]
This implies that $A\in B_1(O_\delta)$. By Lemma~\ref{3.8} and the continuity of index, $A\in B_1(\Omega_A)$, where $\Omega_A$ is the connected component of $\rho_{\text{s-F}}(A)$ containing the origin.
\end{proof}

For more properties related to Cowen-Douglas operators defined on Hilbert spaces, we refer the reader to \cite{CFJ}, \cite{HJ}, \cite{JJ}, \cite{Lin} and \cite{JW} .

\section{A bi-orthogonal system }
To begin with this section, we introduce some new notations. Recall that for a subset $A\subset X$, $[A]$ denotes the closure of span$\{A\}$.
 Assume that $T\in\mathcal{B}(X)$ is invertible. For each fixed $x_0\in X$ and each integer $k\geq0$, we write
 \[\mathfrak N_{T,-k}(x_0)=[T^{-n}x:\, n\geq k],\;\mathfrak N_{T,-\infty}(x_0)=\bigcap_{k=1}^\infty\mathfrak{N}_{T,-k}(x_0).\] If there is no confusion, we  denote them as   $\mathfrak N_{-k}$ and $\mathfrak N_{-\infty}$ for all integers $k\geq0$. In particular, we simply write
 \begin{equation}\label{4.0}\mathfrak N\; (=\mathfrak N_{-1}(x_0))\; =[T^{-n}x_0:\, n\geq 1],\;\;\mathfrak N_{-\infty}=\bigcap_{k=1}^\infty\mathfrak{N}_{-k}.\end{equation}

Keeping the notations above in mind, we state and prove the following result.
\begin{lemma}\label{4.1}
Let $T\in\mathcal{B}(X)$ be  invertible.  Then for each fixed $x_0\in X$, \begin{equation}\label{4.1'}\mathfrak N_{-\infty}\in{\rm Lat}T.\end{equation}
\end{lemma}
\begin{proof}
 Note that $\mathfrak N_{-k}\supset\mathfrak N_{-(k+1)}$ for all $\;k\in\N$.  Continuity and invertibility of $T$ imply \[T(\mathfrak N_{-k})= \mathfrak N_{-k+1},\;{\rm for\;all}\;k\in\N.\]
 Consequently,
\[T\Big(\mathfrak N_{-\infty}\Big)\subset\mathfrak N_{-k}, \;{\rm for\;all}\;k\in\N.\]
Therefore, (\ref{4.1'}) follows.
\end{proof}

\begin{lemma}\label{4.2}
Suppose $T\in\mathfrak B(X)$ is an invertible transitive  operator and $0\neq x_0\in X$. If  $\mathfrak N\neq X$, then  $\{\mathfrak N_{-k}\}_{k=1}^\infty$ is a strictly decreasing sequence of closed subspaces with $\mathfrak N_{-\infty}=\{0\}$.
\end{lemma}
\begin{proof} Since  $T\in\mathfrak B(X)$ is  transitive, $\mathfrak C(T)=X\setminus\{0\}$.  Therefore, $${\rm Lat}T=\{0, X\}.$$
By Lemma~\ref{4.1}, for every $x_0\in X$, $\cap_{k=1}^\infty \mathfrak{N}_{-k}\in{\rm Lat}T=\{0, X\}$. $\mathfrak N_{-\infty}(x)=\{0\}$ follows from $\mathfrak N_{-\infty}\subset \mathfrak N\neq X$.

Suppose that $\{\mathfrak N_{-k}\}$ is not strictly decreasing. Then $\mathfrak N_{-k}=\mathfrak N_{-k+1}$ for some $k\geq 1$. This entails that \[T\big(\mathfrak N_{-k}\big)=\mathfrak N_{-k+1}=\mathfrak N_{-k}.\]
 Since $x_0\neq0$, it follows that $\{0\}\neq\mathfrak N_{-k}\; (\subset\mathfrak N\neq X)$ is a nontrivial invariant subspace of $T$.
 \end{proof}
Recall that a sequence $(x_n,x^*_n)\subset X\times X^*$ is said to be a bi-orthogonal system provided $\langle x_i,x^*_j\rangle=\delta_{ij}$ for all $i,j\in\N$.
\begin{lemma}\label{4.3}
Suppose $T\in\mathfrak B(X)$ is an invertible transitive  operator and $0\neq x_0\in X$. If  $\mathfrak N\neq X$, then there is a sequence $\{x^*_n\}\subset X^*$ so that
$$\{(x_n,x^*_n)\}\subset X\times X^*$$ is a bi-orthogonal system, where $x_n=T^{-n}x_0$ for all $n\in\N$.
\end{lemma}
\begin{proof} Since $T$ is  invertible transitive, it is easy to observe that $\{x_n\}$ is a linearly independent set of $X$.  Note that for every $n\in\N$, $\mathfrak N_{-n}=[x_k]_{k\geq n}$, and that
 $[x_k]_{1\leq k\leq n}$ is a finite dimensional subspace of $X$. Then
\begin{equation}\label{4.4}
\mathfrak N=[x_k]_{k\geq 1}=[x_k]_{1\leq k\leq n}+[x_k]_{k\geq n+1}=[x_k]_{1\leq k\leq n}+\mathfrak N_{-n-1}(x).\end{equation}
Since $T$ is  invertible transitive with $\mathfrak N\neq X$, Lemma \ref{4.2} entails that $\big\{[x_k]_{k\geq n}\big\}_{n=1}^\infty$ is a strict decreasing sequence of closed subspaces with \begin{equation}\label{4.5}\mathfrak N_{-\infty}=\bigcap_{n=1}^\infty[x_k]_{k\geq n}=\{0\}.\end{equation}
Strict decreasing monotonicity of $\big\{[x_k]_{k\geq n}\big\}_{n=1}^\infty$ and linear independence of $\{x_n\}$  imply that $\big\{[x_k]_{1\leq k\leq n}\big\}$ is a strict increasing sequence and    turn (\ref{4.4}) into
\begin{equation}\label{4.6}
\mathfrak N=[x_k]_{1\leq k\leq n}\oplus[x_k]_{k\geq n+1},\;{\rm for\;all\;}n\in\N.\end{equation}
Therefore, \[x_n\notin[x_k]_{1\leq k\leq n-1}\oplus[x_k]_{k\geq n+1}\equiv X_n,\;{\rm for\;all\;}n\in\N.\]
Note that for each $n\in\N$, $X_n$ is a closed subspace of $X$. Then by separation theorem of convex sets there is $x^*_n\in X^*$ so that
\[\langle x^*_n,x_n\rangle=1,\;\langle x^*_n,x\rangle=0, \;{\rm for\;all\;}x\in X_n.\]
\end{proof}
\section{Four operators}

For an invertible transitive operator $T\in \mathfrak B(X)$, If $0\neq x_0\in X$, so that $\mathfrak N\equiv[T^{-k}x_0]_{k\geq1}\neq X$, then by Lemma \ref{4.3}, there is a bi-orthogonal system  $\{(x_n,x^*_n)\}\subset X\times X^*$ with $x_n=T^{-n}x_0$ for all $n\in\N$.  Please keep this in mind. The  assumption and  conclusion will play an important role in the proofs of Theorems \ref{1.4}--\ref{1.7}.

With the same assumption on the operator $T$, the vector $x_0$ and the space $\mathfrak N$ as above, we know that  the following two operators $U, V\in\frak{B}(\mathfrak N)$ are well-defined, where $\mathfrak N=[x_n]=[T^{-n}x_0]_{n\geq1}$.
\begin{equation}\label{4.7'}
U\equiv T^{-1}|_\mathfrak N,\;\; V=(T-x_0\otimes x_1^*)|_\mathfrak N.
\end{equation}
Next,  we  define a closed subspace $\mathfrak N^\dag$ as follows.
\begin{equation}\label{4.8'}
\mathfrak N^\dag=[x_n^*]_{n\geq1},\;\mbox{the  norm closure of span}\{x^*_n\}\;\mbox{in}\; \mathfrak N^*.
\end{equation}
Note $\mathfrak N^*=[x^*_n]^{w^*}$, i.e., the dual $\mathfrak N^*$ of $\mathfrak N$ is just the $w^*$-closure of $\mathfrak N^\dag$. We further define two operators $U^\dag, V^\dag: \mathfrak N^\dag\rightarrow \mathfrak N^*$ below.
\begin{equation}\label{4.9'}
U^\dag=U^*|_{\mathfrak N^\dag},\;\;V^\dag=V^*|_{\mathfrak N^\dag}.
\end{equation}
Then we have the following result.

\begin{lemma}\label{4.10'}
Suppose that $T\in \mathfrak B(X)$ is an invertible transitive operator, $0\neq x_0\in X$, and $\mathfrak N\neq X$. Let $U, V, U^\dag$ and $V^\dag$ be defined by (\ref{4.7'}) and (\ref{4.9'}). Then

 i)  $\mbox{ker}U=\{0\},\;\mbox{ran}U=[x_k]_{k\geq2}$; consequently, $\mbox{ind}U=-1$;

 ii) $\mbox{ker}V=[x_1],\;\mbox{ran}V=\mathfrak N$; consequently, $\mbox{ind}V=1$;

 iii) $VU=I_{\mathfrak N}$,\;$U^*V^*=I_{\mathfrak N^*}$;

 iv) $\mathfrak N^\dag$ is a $w^*$-dense subspace of $\mathfrak N^*$;

 v) $U^\dag, V^\dag\in\mathfrak B(\mathfrak N^\dag)$
 satisfy
 \[\mbox{ker}U^\dag=[x_1^*],\;\mbox{ran}U^\dag=\mathfrak N^\dag,\;\mbox{consequently,\;ind}U^*=ind U^\dag=1;\]
 \[\mbox{ker}V^\dag=\{0\},\;\mbox{ran}V^\dag=[x_n^*]_{n\geq2},\;\mbox{therefore,\;ind}V^*=ind V^\dag=-1;\]
 in particular,
 \[I_{{\mathfrak N}^\dag}=(U^*V^*)|_{{\mathfrak N}^\dag}=U^*|_{{\mathfrak N}^\dag}V^*|_{{\mathfrak N}^\dag};\]

 vi) $x_1=T^{-1}x_0$ is a cyclic vector of $U$; $x^*_1$ is a cyclic vector of $V^\dag$;

vii) $V\in B_1(\Omega_1)$ and $U^\dag\in B_1(\Omega_2)$, where $\Omega_1$ is the connected component of $\rho_{\text{s-F}}(V)$ containing the origin and $\Omega_2$ is the connected component of $\rho_{s-F}(U^\dag)$ containing the origin. Furthermore $\Omega_1,\Omega_2$ are maximal, and $\partial\Omega_1\subseteq \sigma_e(V)$, $\partial \Omega_2\in\sigma_e(U^\dag)$.
\end{lemma}
\begin{proof}
i)-- iii)\; follow directly from (\ref{4.7'}) the definition of $U$ and $V$.

iv)\; Since $\{(x_n,x^*_n)\}\subset X^*$ is a total and complete bi-orthogonal system, i.e., $[x_n]=\mathfrak N$ and $\{x^*_n\}$ separates points of $\mathfrak N$, by separation theorem of convex sets in locally convex spaces, it follows.

v) \& vi)\; $\mbox{ker}U=\{0\}$ and $\mbox{ran}U=[x_k]_{k\geq2}$ in i) entail that $U^*\in\mathfrak B(\mathfrak{N}^*)$ is surjective with
\[\mbox{ker}U^*=(\mbox{ran}U)^\bot=[x_k]_{k\geq2}^\bot=[x^*_1].\]
\[\langle U^*x^*_i,x_j\rangle=\langle x^*_i,Ux_j\rangle=\langle x^*_i,x_{j+1}\rangle=\delta_{i,j+1},\;i,j\in\N\]
deduce that
\[U^*x_n^*=x^*_{n-1},\;\mbox{for all }n\geq2.\]
Therefore, $U^\dag=U^*|_{\mathfrak N^\dag}\in\mathfrak B(\mathfrak{N}^\dag)$ with $\mbox{ind}U^*=1$.

$\mbox{ker}V=[x_1]$ and $\mbox{ran}V=\mathfrak N$ in ii) imply that $V^*\in\mathfrak B(\mathfrak{N}^*)$ is injective with
\[\mbox{ran}V^*=(\mbox{ker}V)^\bot=[x_1]^\bot=[x^*_k]^{w^*}_{k\geq2}.\]
While $V^*x^*_n=x^*_{n+1}$ for all $n\in\N$ deduce that $V^\dag=V^*|_{\mathfrak N^\dag}\in\mathfrak B(\mathfrak{N}^\dag)$ with $\mbox{ind}V^*=-1$,
and further, that  $x_1^*$ is a cyclic vector of $V^\dag\in\mathfrak B(\mathfrak N^\dag)$. Clearly, $x_1$ is a cyclic vector of $U$. Therefore, v) and vi) have been shown.

{ vii) follows from ii), iii), v), vi), Lemma~\ref{3.10} and Lemma~\ref{3.13}.}
\end{proof}

\section{A lifting property}
In this section, we begin with the following simple property.
\begin{proposition}\label{4.7}
Let $T\in\mathfrak B(X)$. Then

i)  $T$ is invertible if and only if $T^*$ is invertible;

ii) $\sigma(T^*)=\sigma(T).$

iii) If $X$ is a complex Hilbert space, then the equality ii) should be replaced by
$\sigma(T^*)=\overline{\sigma(T)}\equiv\{\bar{\lambda}: \lambda\in\sigma(T)\}$.
\end{proposition}
\begin{proof}
i) It suffices to note  $T\in\mathfrak B(X)$ is invertible if and only if $T$ is bijective, and which is equivalent to $T^*: X^*\rightarrow X^*$ is bijective.

ii) follows from i) and the definition of the conjugate operator $T^*$ of $T$.
\end{proof}

We need also the following ``lefting property" related to $w^*$-to-$w^*$ continuous operators defined a subspace of the dual $X^*$ of $X$.
\begin{theorem}\label{4.8}
 Suppose that $Z^\dag\subset Z^*$ is { a separable Banach space which is} $w^*$- dense subspace of $Z^*$, and $A^\dag\in\mathfrak B(Z^\dag)$ is $w^*$-continuous. Then

i)\; $A^\dag$ has a $w^*$-continuous extension $A^*\in \mathfrak B(Z^*)$;

ii) \; { If $A^\dag\in\mathfrak B(Z^\dag)$ is invertible, then $A^*\in \mathfrak B(Z^*)$ is also invertible. Furthermore, if $A^\dag\in B_1(\Omega)$, then $A^\dag\in\mathfrak B(Z^\dag)$ is invertible if and only if $A^*\in \mathfrak B(Z^*)$ is invertible.} Consequently,

iii) { \[\sigma(A^\dag)\supseteq \sigma(A^*),\;{\rm and}\;\;\rho(A^\dag)\subseteq \rho(A^*);\] Furthermore,  if $A^\dag\in B_1(\Omega)$, then \[\sigma(A^\dag)= \sigma(A^*), {\rm and}\; \;\rho(A^\dag)=\rho(A^*).\]}
\end{theorem}
\begin{proof} It suffices to show  ii) and iii).

ii)\;Since $A^\dag\equiv A^*|_{Z^\dag}\in\mathfrak B(Z^\dag)$ is invertible, $(A^\dag)^{-1}: Z^\dag\rightarrow Z^\dag$ is an isomorphism.
We first claim that $(A^\dag)^{-1}$ is also a b-$w^*$-isomorphism of $Z^\dag$, i.e., it is   continuous with respect to the bounded weak-star topology of $Z^*$. Since $Z^\dag$ is separable,
the b-$w^*$ topology is metrizable on each bounded subset of $Z^\dag$. Consequently, the $w^*$-convergence on bounded sets is equivalent to $w^*$-sequential convergence.
Let $\{x^*_n\}\subset Z^\dag $ be a (bounded) sequence $w^*$-converging to $x^*\in Z^\dag$. Then invertibility of $A^\dag$ on $Z^\dag$ implies that $\{y^*_n\}$ is also
 bounded in $Z^\dag$, where $y^*_n=A^\dag x^*_n$. The Banach-Alaoglu theorem says that there is a subsequence of $\{y^*_n\}$ (again denoted by $\{y^*_n\}$)
$w^*$-converging to some $y^*\in Z^*$. We want to show that $y^*=A^\dag x^*$. Since $A^\dag$ is $w^*$-continuous, we obtain $x^*_n=A^\dag y^*_n$ $w^*$-converges to $A^\dag y^*$.
Therefore, $A^\dag y^*=x^*$. Consequently,  $(A^\dag)^{-1}x^*=y^*\in Z^\dag$.

 We will show that $(A^\dag)^{-1}$ can be b-$w^*$-continuously extended to $Z^*$. It is equivalent to that for each $z^*\in Z^*$, there exists $x^*\in Z^*$ such that for every bounded  sequence $\{z^*_n\}$ in $Z^\dag$ which is $w^*$-converging to $z^*\in Z^*$, we have $w^*$-$\lim_n(A^\dag)^{-1}z^*_n=x^*$. Indeed, suppose that $\{u^*_n\}$ and $\{v^*_n\}$ are two bounded sequences in $Z^\dag$ so that both of them are $w^*$-converging to $z^*\in Z^*$. Then $\{w_n^*\}$ \;($w^*_n=u^*_n-v^*_n$)\;is $w^*$-converging to $0$. Therefore, $w^*$-continuity of $(A^\dag)^{-1}$ on $Z^\dag$ entails that  $(A^\dag)^{-1}w^*_n$ is again $w^*$-converging to $0$. This says that the two sequences $\{(A^\dag)^{-1}u_{n}^*\}$ and $\{(A^\dag)^{-1}v_{n}^*\}$ have  same $w^*$-cluster points. Relative $w^*$-compactness of them deduces that $w^*$-$\lim_n(A^\dag)^{-1}u^*_n=w^*\text{-}\lim_n(A^\dag)^{-1}v^*_n$. Therefore, we have shown that $(A^\dag)^{-1}$ can be b-$w^*$-continuously extended to the b-$w^*$ completion of $Z^\dag$. By the Krein-\u{S}mulian theorem, the b-$w^*$ completion of $Z^\dag$ is just $Z^*$ and both the b-$w^*$ continuity and the $w^*$-continuity coincide on $Z^*$ for all linear operators in $\mathfrak B(Z^*)$. We denote the $w^*$-continuous extension of $A^\dag$ by $B^*$. For completion of the proof of v), it remains to show that $A^*B^*=I_{Z^*}$.
 Given $z^*\in Z^*$, let $\{z_n^*\}\subset Z^\dag$ be a sequence $w^*$-converging to $z^*$. Then $w^*$-continuity of both $A^*$ and $B^*$ entails that
 \[(A^*B^*)z^*=w^*\text{-}\lim_n(A^*B^*)z^*_n=w^*\text{-}\lim_n(A^*B^*)|_{Z^\dag}z^*_n=w^*\text{-}\lim_nz^*_n=z^*.\]
 Therefore, $A^*B^*=I_{Z^*}$ and (1) has been shown.

{Suppose that $A^\dag\in B_1(\Omega)$ and that $A^*\in \mathfrak B(Z^*)$ is invertible. Then there is  $K>0$ such that $\|A^*x\|\geq K\|x\|$ for all $x\in Z^*$. In particular, $\|A^\dag x\|\geq K\|x\|$ for all $x\in Z^\dag$. Thus, ${\rm ran} A^\dag$ is closed. By Lemma~\ref{3.8}, ${\rm ran} A^\dag=Z^\dag$ and $A^\dag$ is invertible.}

{iii) By ii) we have just proven,  if $A^\dag-\lambda\in\mathfrak B(Z^\dag) $ is invertible on $Z^\dag$, then  its natural extension $A^*-\lambda\in \mathfrak B(Z^*)$ is invertible. Equivalently, if
 $\lambda\in\rho(A^\dag)$, then $\lambda\in\rho(A^*)$. Consequently, $\rho(A^\dag)\subseteq \rho(A^*)$.} Suppose $A^\dag\in B_1(\Omega)$. By ii), $A^\dag-\lambda\in\mathfrak B(Z^\dag) $ is invertible on $Z^\dag$ if and only if  its natural extension $A^*-\lambda\in \mathfrak B(Z^*)$ is invertible. Equivalently, if
 $\lambda\in\rho(A^\dag)$, then $\lambda\in\rho(A^*)$. Consequently, $\rho(A^\dag)=\rho(A^*)$.

\end{proof}
\section{Spectrum properties}

\begin{lemma} \label{4.9}
Suppose that $T\in\mathfrak B(X)$ is invertible transitive. Let the spaces $\mathfrak N$, $\mathfrak N^\dag$, and the  operators $U,V\in\mathfrak B(\mathfrak N)$ and $U^\dag, V^\dag\in\mathfrak B(\mathfrak N^\dag)$ be defined in Lemma \ref{4.10'}. Then
\begin{enumerate}
\item  $\sigma(U^\dag)=\sigma(U^*)\;\;\;{\rm and }\;\rho(U^\dag)=\rho(U^*)$;
\item  $\sigma(V^\dag)\supseteq\sigma(V^*)\;\;\;{\rm and }\;\rho(V^\dag)\subseteq\rho(V^*)$;
\item  $\sigma_e(U^\dag)=\sigma_e(U^*).$
\end{enumerate}
\end{lemma}
\begin{proof}
{ (1) and (2) follow directly from Lemma \ref{4.10'} and Theorem \ref{4.8}. By Lemma~\ref{3.9} and Lemma \ref{4.10'}, $U^\dag\in B_1(\Omega_2)$, $\Omega_2$ is maximal and $$\partial \Omega_2\subseteq \sigma_e(U^\dag).$$  By Lemma~\ref{4.10'}, $x_1=T^{-1}x_0$ is a cyclic vector of $U$. By Proposition~\ref{3.11}, ${\rm dimker} (U^*-\lambda)\leq 1$ for all $\lambda$ and ${\rm dimker} (U^\dag-\lambda)\leq 1$ for all $\lambda$.

By Lemma~\ref{3.3},  for any $\lambda\in\mathbb C$, $U^\dag-\lambda$ has dense range.  To prove (3), it suffices to show that surjectivity of $U^\dag-\lambda$ on $Z^\dag$ implies surjectivity of $U^*-\lambda$ on $Z^*$.  Let $z^*\in Z^*$, and  $\{z_n^*\}\subset Z^\dag$ be a sequence such that
\[
z^*=w^*\text{-}lim_n z_n^*.
\]
Let $Y^\dag=\ker(U^\dag-\lambda)\subset Z^\dag $. Define $Q:\, Z^\dag/Y^\dag\rightarrow Z^\dag$ by
\[
Q(y^*+Y^\dag)=(U^\dag-\lambda)y^*, \;\; y^*\in Z^*.
\]
Then $Q$ is an isomorphism from $Z^\dag/Y^\dag$ onto $Z^\dag$. Therefore, for each $z_n^*$, there exists a $y_n^*\in Z^\dag$ such that
\[
Q(y_n^*+Y^\dag)=(U^\dag-\lambda)y_n^*=z_n^*.
\]
Since, $\{y_n^*+Y^\dag\}$ is a bounded sequence in $Z^\dag/Y^\dag$, we may assume that $\{y^*_n\}$ is  bounded in $Z^*$. By the Banach-Alaoglu theorem, there is a  subsequence of $\{y^*_n\}$ (again denoted by $\{y^*_n\}$)  $w^*$-convergent to some point $y^*\in Z^*$.
Then
\[
z^*=w^*\text{-}\lim_n z_n^*=w^*\text{-}\lim_n (U^\dag-\lambda)y_n^*=w^*\text{-}\lim_n(U^*-\lambda)y_n^*=(U^*-\lambda)y^*.
\]
Therefore, $U^*-\lambda$ is surjective.
}
\end{proof}

\begin{lemma}\label{4.10}
Suppose that $T\in \mathfrak B(X)$ is  invertible transitive, $0\neq x_0\in X$, and that the spaces $\mathfrak N$, $\mathfrak N^\dag$, and the  operators $U,V\in\mathfrak B(\mathfrak N)$ and $U^\dag, V^\dag\in\mathfrak B(\mathfrak N^\dag)$ be defined in Lemma \ref{4.10'}.
Assume that  $\Omega_0$ is the connected component of $\mathbb C\setminus \sigma(T^{-1})$ containing the origin.
 Then

 i)\; For every $\lambda\in\Omega_0$, $U-\lambda$ is a Fredholm operator with $$\text{dimker}(U-\lambda)=0, \;\;{\rm and\;}\;\text{co-dim}(U-\lambda)=1.$$ Consequently,  \[\text{ind}(U-\lambda)\equiv \text{dimker}(U-\lambda)-\text{co-dim}(U-\lambda)=-1.\]

 ii)\; $V\in B_1(\Omega_V)$, where $\Omega_V=\left\{\lambda:\,\frac{1}{\lambda}\in\mathbb C\setminus \sigma(U)^\land \bigcup\{\infty\}\right\}$ and $\partial \Omega_V\subseteq \sigma_e(V)$.
\end{lemma}
\begin{proof} i)  By Lemma \ref{4.10'}, $\mbox{dim ker}U=0$, $\mbox{co-dim ran}U=1$, and  ind $U=-1$.
Since $T\in\mathfrak B(X)$ is invertible transitive, for every $\lambda\in\Omega_0$, $T^{-1}-\lambda$ is again invertible.
Therefore, $U-\lambda=(T^{-1}-\lambda)|_{\mathfrak N}$ is injective, and consequently, $\mbox{dim ker}(U-\lambda)=0$. Since $\Omega_0$ is the connected component of $\mathbb C\setminus \sigma(T^{-1})$ containing the origin, and since ind $U=-1$, continuity of the index entails ind$(U-\lambda)=-1$ and $\mbox{co-dim ran}(U-\lambda)=1$.

ii)
 By Lemma \ref{4.10'}, $\mbox{dimker}V=1$, $\mbox{co-dimran}U=0$, ind$V=1$ and  $VU=I$.
Write $e_0=T^{-1}x_0$ and $\delta=\frac{1}{\|U\|}$. Set
\[S_\lambda=\sum_{n=0}^\infty \lambda^nU^n\in \mathfrak B(\mathfrak N),\;{\rm for\;all\;}|\lambda|<\delta.\] Then $S_\lambda$ is holomorphic on $\{\lambda:\,|\lambda|<\delta\}$. Since $V$ is surjective and since
$VS_\lambda e_0=\lambda S_\lambda e_0$ for all $|\lambda|<\delta$, we obtain $V\in B_1(\Omega_V)$. It follows from uniqueness of the extension of the holomorphic map $S_\lambda$. Note that $\Omega_V$ is maximal. Then by Lemma~\ref{3.9}, $\partial \Omega_V\subseteq \sigma_e(V)$ .
\end{proof}
\begin{lemma}\label{4.11}
Suppose $T\in \mathfrak B(X)$ is transitive invertible and $U,V$ be defined by Lemma~\ref{4.10'}. Then $$\rho_{\text{s-F}}(U)=\rho_{\text{F}}(U).$$ { Furthermore,

i) $\sigma_e(U)\subseteq\sigma_e(T^{-1})$;

ii) $\sigma_e(V)=\left\{\frac{1}{\lambda}:\,\lambda\in \sigma_e(U)\right\}$;

iii) $\sigma_e(V^\dag)=\left\{\frac{1}{\lambda}:\,\lambda\in \sigma_e(U^\dag)\right\}$.}
\end{lemma}
\begin{proof}
Since $T$ is transitive, $\sigma_p(T^{-1})=\emptyset$. Note that $U=T^{-1}|_{\mathfrak N}$. Then $\sigma_p(U)=\emptyset$. By Lemma~\ref{3.10} and Lemma~\ref{4.10'},  $x_1=T^{-1}x_0$ is a cyclic vector of $U$. It follows  from  Proposition~\ref{3.11} that \[\mbox{co-dimran}(U-\lambda)=\mbox{dimker}(U^*-\lambda)\leq 1\] for all $\lambda\in \mathbb C$. Thus, $\rho_{\text{s-F}}(U)=\rho_{\text{F}}(U)$.

Suppose $\lambda\in \sigma_e(U)$. Then $\text{ran}(U-\lambda)$ is not closed. Therefore, there exists a sequence of unit vectors $\xi_n\in \mathfrak N$ such that
$$\lim_n \|(U-\lambda) \xi_n\|=0.$$ Consequently, $$\lim_n\|(T^{-1}-\lambda)\xi_n\|=0.$$ This implies that $\lambda\in \sigma_e(T^{-1})$. Thus i) holds.

{ Note that
\[
\pi(UV)=\pi(VU)=\pi(I_{\mathfrak N})
\]
and that
\[
\pi(U^\dag V^\dag)=\pi(V^\dag U^\dag)=\pi(I_{\mathfrak N^\dag}).
\] Then we obtain ii) and iii).}
\end{proof}
\begin{lemma}\label{4.12}
Suppose $T\in \mathfrak B(X)$ is an invertible transitive operator. Let the Banach space $\mathfrak N$ and the operator $U$ be defined in Lemma~\ref{4.10'}. If $\Omega'$ is a connected component of $\rho_F(U)$  not containing the origin, then $\Omega'\subset\rho(U)$ and $\partial \Omega'\subseteq\sigma_e(U)$.
\end{lemma}
\begin{proof}
 By Lemma~\ref{3.10} and Lemma~\ref{4.10'}, $T^{-1}x_0\in C(U) (\neq\emptyset)$. It follows from Proposition~\ref{3.12} and Lemma \ref{4.11} that $\Omega\equiv\rho_{{\rm F}}^{-1}(U)=\rho_{\text{s-F}}^{-1}(U)$ is simply connected.  Due to Lemma~\ref{4.10}, $\Omega$ is a connected component of $\rho_{\text{F}}(U)$ containing the origin.  Since $\Omega'$ is a connected component of $\rho_F(U)$  not containing the origin,  $\Omega'\bigcap\Omega=\emptyset$. By Proposition~\ref{3.11}, $\mbox{co-dimran}(U-\lambda)\leq 1$ for all $\lambda\in \Omega'$. $\Omega'\bigcap \Omega=\emptyset$ deduces $\mbox{co-dimran}(U-\lambda)=0$ for all $\lambda\in \Omega'$. Since $\sigma_p(T^{-1})=\emptyset$ and since $\mathfrak N\in \mbox{Lat} T^{-1}$, $\sigma_p(U)=\emptyset$. Therefore, $\mbox{ran}(U-\lambda)=\mathfrak N$ for all $\lambda\in \Omega'$, and  $\Omega'\subset\rho(U)$ follows.
\end{proof}

\begin{lemma}\label{4.13}
Suppose that $\Omega_0$ is the connected component of $\sigma(U)^\land\bigcap\rho_{\text{F}}(U)$ containing the origin, and that $\Omega_U$ is another connected component of $\sigma(U)^\land\bigcap\rho_{\text{F}}(U)$ which is different from $\Omega_0$. Then
\begin{enumerate}
\item $\mbox{codim ran}(U-\lambda)=1$, $\mbox{ind}(U-\lambda)=-1$, $\mbox{for all}\; \lambda\in \Omega_0$;
\item $\partial\Omega_0,\; \partial \Omega_U\subset \sigma_e(U)$,\; \text{and}\; $\Omega_U\subseteq \rho(U)$;
\item $\Omega_V\equiv\{\frac{1}{\lambda}:\lambda\in \mathbb C\setminus \sigma(U)^\land\bigcup\{\infty\}\}$ is the connected component of $\sigma(V)^\land\bigcap \rho_F(V)$ containing the origin, and $\Omega_U^{-1}=\{\frac{1}{\lambda}:\,\lambda\in \Omega_U\}$ is another connected component of $\sigma(V)^\land\bigcap\rho_F(V)$, and
\item $\partial\Omega_V, \partial\Omega_U^{-1}\subset \sigma_e(V).$
\end{enumerate}
\end{lemma}
\begin{proof}
By Lemma~\ref{4.10'} and Lemma~\ref{4.9}, $U^\dag\in B_1(\Omega_0)$ and $\Omega_0$ is maximal. By Lemmas~\ref{4.10'}, \ref{4.9} and ~\ref{4.11}, we obtain (1) and (2).

Now, we are in position to prove (3). Clearly, $\Omega_V$ is the connected component of of $\sigma(V)^\land\cap \rho_F(V)$ containing the origin. Hence, $\partial\Omega_V\subseteq \sigma_e(V)$. By Lemma~\ref{4.12}, $\Omega_U\subseteq \rho(U)$ and $\partial \Omega_U\subseteq \sigma_e(U)$. Note that \[
\pi(UV)=\pi(VU)=\pi(V)\pi(U)=I.
\]
Then $\Omega_U^{-1}$ is another connected  component of $\sigma(V)^\land\cap\rho_F(V)$ with  $$\partial \Omega_U^{-1}\subseteq \sigma_e(V).$$
\end{proof}

\section{Proof of Theorem 1.4}
In this  section, we will prove Theorem \ref{1.5} presented in the first section. Before doing it, we should emphasize that the operators $U,V,U^\dag$ and $V^\dag$ are defined in Lemma~\ref{4.10'}. Now, we restate Theorem \ref{1.5} as follows. 
\begin{theorem}\label{4.14}
Suppose that $T\in \mathfrak B(X)$ is an invertible operator and there exists $0\neq x_0\notin C(T^{-1})$. If $\sigma(U)^\land\bigcap \rho_\text{F}(U)$ has a connected component not containing the origin, then $T$ is intransitive.
\end{theorem}
\begin{proof}

{ Suppose, to the contrary, that $T$ is transitive. 
Let $\Omega_0$ be the connected component of $\sigma(U)^\land\bigcap\rho_{\text{F}}(U)$ containing the origin,  $\Omega_U$ be another connected component of $\sigma(U)^\land\bigcap\rho_{\text{F}}(U)$ which is different from $\Omega_0$, and let $\Omega_0^\dag$ be the connected component of $\sigma(U^\dag)\bigcap \rho_F(U^\dag)$ containing the origin. Then by Lemma~\ref{4.10'}, $U^\dag\in B_1(\Omega_0^\dag)$. We first show
\begin{equation}\label{9.2'}\Omega_0^\dag=\Omega_0.\end{equation}
By Theorem~\ref{4.8} and Lemma~\ref{4.9},
\[
\sigma(U^\dag)^\land\cap \rho_F(U^\dag)=\sigma(U^*)^\land\cap\rho_F(U^*)=\sigma(U)^\land\cap\rho_F(U)
\]
and
\[
\sigma_e(U^\dag)=\sigma_e(U^*).
\]
Therefore, (\ref{9.2'}) holds. Consequently, $\Omega_U\bigcap \Omega_0^\dag=\Omega_U\bigcap\Omega_0=\emptyset$.

By Lemma~\ref{4.13}, $\Omega_U\subseteq \rho(U)$ and $\Omega_V\equiv\big\{\frac{1}{\lambda}:\lambda\in \mathbb C\setminus \sigma(U)\bigcup\{\infty\}\big\}$ is the connected component of $\sigma(V)\bigcap \rho_F(V)$ containing the origin and
\[
\Omega_U^{-1}\bigcap\Omega_V=\emptyset.
\]
By Lemma~\ref{4.10'}, $V\in B_1(\Omega_V)$ and $x_1^*$ is a cyclic vector of $V^\dag$. By Lemma~\ref{3.12}, $\rho_F^{-1}(V^\dag)$ is simply connected.

Since $V^\dag U^\dag=I_{\mathfrak{N}^\dag}$ and since $\sigma_e(V^\dag)=\left\{\frac{1}{\lambda}:\,\lambda\in\sigma_e(U^\dag)\right\}$,
$\rho_F^{-1}(V^\dag)=\Omega_V$. Therefore,
\begin{equation}\label{9.3'} {\rm ind}(V^\dag-\lambda)=0,\;{\rm for\;all}\; \lambda\in\Omega_U^{-1}.\end{equation}
Next, we show
\begin{equation}\label{9.4'}\Omega_U^{-1}\subseteq \rho(V).\end{equation}

By Theorem~\ref{4.8}, $\rho(V^\dag)\subseteq \rho(V^*)$. Since $\Omega_U^{-1}\subseteq \rho_F(V)$, $\Omega_U^{-1}\subseteq \rho_F(V^*)$. Lemma~\ref{4.10'} entails $V\in B_1(\Omega_V)$. By Lemma~\ref{3.8}, $\text{dimker}(V^*-\lambda)=0$ for $\lambda\in \Omega_U^{-1}$. Since $V^\dag-\lambda=(V^*-\lambda)|_{{\mathfrak N}^\dag}$, $\text{dimker}(V^\dag-\lambda)=0$ for all $\lambda\in\Omega_U^{-1}$.
Note that for $\lambda\in\Omega_U^{-1}$, ${\rm ind}(V^\dag-\lambda)=0$. It follows that $V^\dag-\lambda$ is invertible and
\[\Omega_U^{-1}\subseteq \rho(V^\dag)\subseteq \rho(V^*).\] Thus, (\ref{9.4'}) holds.

 By Lemma~\ref{4.13}, for $\lambda\in \Omega_U$, $U-\lambda$ is invertible. It follows from (\ref{9.4'}) that
\[
V(U-\lambda)=\lambda\left(\frac{1}{\lambda}-V\right)
\]
is invertible. This is a contradiction to that $V\in B_1(\Omega_V)$ is not invertible. }

\end{proof}

\section{More spectrum properties}
Recall (\ref{1.6}) that for $A\in\mathfrak B(X)$,  we can define a quotient operator $A^Y\in\mathfrak B(X/Y)$  by
\begin{equation}\label{5.1}
A^Y(x+Y)=Ax+Y,\;\;{\rm for\;all\;}x\in X.
\end{equation}
if $Y\subset X$ is a closed invariant subspace of $A$.
\begin{lemma}\label{5.2}
Suppose that $A\in\mathfrak B(X)$ and $Y\subset{\rm Lat}A$.   Then $A\in\mathfrak B(X)$ admits the three-operator property, i.e.,  invertibility of any two of the three operators $A, A^Y$ and $A|_Y$  implies  invertibility of the third one. Consequently,
\begin{enumerate}
\item $\sigma(A)\subseteq \sigma(A|_Y)\cup \sigma(A^Y)$;
\item $\sigma(A|_Y)\subseteq \sigma(A)\cup \sigma(A^Y)$;
\item $\sigma(A^Y)\subseteq \sigma(A)\cup \sigma(A|_Y)$.
\end{enumerate}
\end{lemma}
\begin{proof}
Here we only show that if both $A_Y$ and $A^Y$ are invertible, then $A$ is invertible.  Suppose, to the contrary, that $A$ is either not injective, or, not surjective.
The former implies that there is $0\neq z\in X$ so that $Az=0$. If $z\in Y$, then $A_Y$ is not invertible. If $z\in X\setminus Y$, then $z+Y\neq0$, $A(z+Y)=0$, which means that $A^Y$ is not injective. The latter entails that there is $z\in X\setminus\mbox{ran}A$. Consequently, if $z\notin Y$, then $z+Y\notin\mbox{ran}(A^Y)$; if $z\in Y$, then $A|_Yy\neq z$ for all $y\in Y$. Thus, either $A^Y$, or, $A|_Y$ is not surjective.
\end{proof}

\begin{lemma}\label{5.3}
Suppose that $T\in\mathfrak B(X)$ is an invertible transitive operator and $0\neq x_0\in X$ so that  $\mathfrak N\equiv[x_n]_{n\geq1}\subset X$, and  that
$\{(x_n,x^*_n)\}\subset X\times X^*$ is the bi-orthogonal system with $x_n=T^{-n}x_0$ for all $n\in\N$ defined in Lemma \ref{4.3}.
Let $T_1=T-x_0\otimes x^*_1$.
Then
\begin{enumerate}
\item $\mathfrak N\in\mbox{Lat}T_1\bigcap\mbox{Lat}T^{-1}$;
\item $(T^{-1})^\mathfrak NT_1^\mathfrak N=I_{X/\mathfrak N}$;
\item $\mbox{ind}(T^{-1})^\mathfrak N=1$. Consequently, there is a maximal connected open set $\Omega_{\rm max}\subset\mathbb C$ containing the origin such that $(T^{-1})^\mathfrak N\in B_1(\Omega_{\rm max})$.
\end{enumerate}
\end{lemma}
\begin{proof}
(1)  It suffices to note that $T_1x_1=0$ and  $T_1x_k=x_{k-1}$ for all $k\geq 2$.

(2) It follows from (1) that we can define the quotient mappings as following:
\[(T^{-1})^\mathfrak N(x+\mathfrak N)=T^{-1}x+\mathfrak N,\]
\[T_1^\mathfrak N(x+\mathfrak N)=T_1x+\mathfrak N.\]
Therefore, for every $x+\mathfrak N\in X/\mathfrak N$,
\begin{equation}\label{6.4}(T^{-1})^\mathfrak NT_1^\mathfrak N(x+\mathfrak N)=T^{-1}T_1x+\mathfrak N.\end{equation}
This shows that $(T^{-1})^\mathfrak NT_1^\mathfrak N=I_{X/\mathfrak N}$.

(3) Since  $(T^{-1})^\mathfrak N(x_0+\mathfrak N)=x_1+\mathfrak N=\mathfrak N$,  $0\neq x_0+\mathfrak N\in\mbox{ker}(T^{-1})^\mathfrak N$. Note that
\[
T_1^{\mathfrak N}(x_0+{\mathfrak N})=Tx_0+\alpha x_0+{\mathfrak N}.
\]
Then $[x_0+{\mathfrak N}, T_1^{\mathfrak N}(x_0+{\mathfrak N})]=[x_0+{\mathfrak N}, Tx_0+{\mathfrak N}]$. By induction, we know that for all $k\geq 1$,
\[
[x_0+{\mathfrak N}, T_1^{\mathfrak N}(x_0+{\mathfrak N}), \cdots, \left(T_1^{\mathfrak N}\right)^k(x_0+{\mathfrak N})]\]
\[=[x_0+{\mathfrak N}, Tx_0+{\mathfrak N}, \cdots, T^kx_0+{\mathfrak N}].
\]
 Since $T$ is transitive, $x_0+{\mathfrak N}$ is a cyclic vector of $ T_1^{\mathfrak N}$.
By Lemma \ref{3.13}, there is a maximal connected open set $\Omega_{\rm max}\subset\mathbb C$ containing the origin such that $(T^{-1})^\mathfrak N\in B_1(\Omega_{\rm max})$

\end{proof}

\begin{lemma}\label{5.5}
Suppose that $T\in\mathfrak B(X)$ is  invertible transitive  and $0\neq x_0$ in $ X$, and that  the following three conditions hold.

(1) $\sigma(T^{-1}|_\mathfrak N)^\land\bigcap \rho_\text{F}(T^{-1}|_\mathfrak N)$ has a unique connected component;

(2) ${\rm int}\sigma(T^{-1})^\land$ admits two bounded connected components $\Omega_0$ and $\Omega_1$ such that $0\in\Omega_0$ and $\Omega_1\bigcap\rho(T^{-1})\neq\emptyset;$

(3) $\sigma(T^{-1})\subset\sigma((T^{-1})^\mathfrak N)$.

\noindent
Let $\Omega=\Omega_1\bigcap\rho(T^{-1})$. Then

i)  $\Omega\bigcap\sigma(T^{-1}|_\mathfrak N)^\land=\emptyset$;

ii) $\Omega\subset\rho((T^{-1})^\mathfrak N)$ and $\partial\Omega\subset\sigma((T^{-1})^\mathfrak N)$.

iii) $\lambda\in\Omega^{-1}$ implies that $\ker(T_1^\mathfrak N-\lambda)=\{0\}$, where $T_1=T-x_0\otimes x^*_1$ is defined in Lemma \ref{5.3}.

\end{lemma}
\begin{proof}
 i)\; By Lemma~\ref{2.8}, $\sigma(T^{-1}|_\mathfrak N)\subset\sigma(T^{-1})^\land$.
Therefore,\begin{equation}\label{5.7}
\sigma(T^{-1}|_\mathfrak N)^\land\subseteq \sigma(T^{-1})^\land.
\end{equation}
Assume that $\Omega_0^\prime$ is the unique connected component of $\sigma(T^{-1}|_\mathfrak N)^\land\bigcap \rho_\text{F}(T^{-1}|_\mathfrak N)$ containing the origin. Since
\[
\sigma(T^{-1}|_\mathfrak N)^\land\subseteq \sigma(T^{-1})^\land,
\]
$\Omega_0^\prime$ is a subset of $\Omega_0$. Therefore, $\Omega_0^\prime\bigcap\Omega=\emptyset$ follows from $\Omega_0\bigcap \Omega=\emptyset$. Thus, i) is shown.

ii)\;By i), $\Omega\subset\rho(T^{-1}|_\mathfrak N)$. Therefore, $\partial\Omega\subset\sigma((T^{-1})^\mathfrak N)$ follows from Lemma \ref{5.2}  and (3).
 $\sigma(T^{-1})\subset\sigma((T^{-1})^\mathfrak N)$, $\Omega\subset\rho(T^{-1})$,  and Lemma \ref{5.2}  together deduce $\partial\Omega\subset\sigma((T^{-1})^\mathfrak N)$ and $\Omega\subset\rho((T^{-1})^\mathfrak N)$.

iii) \; Let $\lambda\in\Omega^{-1}$ and  $z+\mathfrak N\in\ker(T_1^\mathfrak N-\lambda)$ for some $z\in X$. Then by definition of $T_1^\mathfrak N$, we obtain
\[(T-\lambda)z-\langle x^*_1,z\rangle x_0\in \mathfrak N.\]
Recall that $x_1=T^{-1}x_0$.
Then
\[
T^{-1}\big((T-\lambda)z-\langle x^*_1,z\rangle x_0\big)=T^{-1}(T-\lambda)z-\langle x^*_1,z\rangle x_1\in\mathfrak N.
\]
This and $x_1\in\mathfrak N$ imply that $T^{-1}(T-\lambda)z\in \mathfrak N$. Consequently, $(\frac{1}{\lambda}-T^{-1})z\in \mathfrak N$. By i), $\frac{1}{\lambda}\in \rho(T^{-1})$, $\frac{1}{\lambda}\in \rho(T^{-1}|_\mathfrak N)$. There exits  $z'\in \mathfrak N$ such that
\[
(\frac{1}{\lambda}-T^{-1})z=(\frac{1}{\lambda}-T^{-1})z'.
\]
This entails that $(\frac{1}{\lambda}-T^{-1})(z-z')=0$, and further that  $z=z'\in \mathfrak N$. Hence, $\ker(T_1^\mathfrak N-\lambda)=\{0\}$.
\end{proof}
\section{Proof of Theorem 1.5}
We restate Theorem \ref{1.7} as follows.
\begin{theorem}\label{5.8}
Let $T\in \mathfrak B(X)$ be invertible and ${\rm int}\sigma(T^{-1})^\land$ contain  two components $\Omega_0$,  $\Omega_1$ with $0\in\Omega_0$  and $\Omega_1\bigcap\rho(T^{-1})\neq\emptyset$. If $T^{-1}$ admits a nonzero noncyclic vector  $x_0\in X$ such that $\sigma(T^{-1})\subset\sigma((T^{-1})^\mathfrak N)$, then $T$ is intransitive.
\end{theorem}
\begin{proof}
Suppose, to the contrary, that $T$ is transitive. Since $0\neq x_0\in X$ is a noncyclic vector, $\mathfrak N=[T^{-n}x_0]_{n\geq1}\neq X$. By Lemma~\ref{3.9} and (3) in Lemma \ref{5.3}, there is a maximal connected open set $\Omega_{\rm max}\subset\mathbb C$ containing the origin  such that $$(T^{-1})^\mathfrak N\in B_1(\Omega_{\rm max}).$$ It follows from $\sigma(T^{-1})\subset\sigma((T^{-1})^\mathfrak N)$ and Lemma \ref{5.1} that
 $\Omega_{\rm max}\subset\Omega_0$. If  $\sigma(T^{-1}|_\mathfrak N)^\land\bigcap\rho_{\rm F}(T^{-1}|_\mathfrak N)$ contains two connected components, then by Theorem \ref{4.14}, $T$ is intransitive.

 Now, assume that  $\sigma(T^{-1}|_\mathfrak N)^\land\bigcap\rho_{\rm F}(T^{-1}|_\mathfrak N)$ contains just one connected component $\Omega=\Omega_0\bigcap\rho(T^{-1})$.
Since $\Omega_1\bigcap\rho(T^{-1})\neq\emptyset$, by Lemma \ref{5.5}, we obtain $\Omega_1\subset \rho\big((T^{-1})^\mathfrak N\big)$.  It follows from Lemma \ref{5.3} and Proposition \ref{3.10} that
\[(T^{-1})^\mathfrak NT_1^\mathfrak N=I_{X/\mathfrak N},\;\;{\rm and\;}\; C(T_1^\mathfrak N)\neq\emptyset.\]
Applying Propositions \ref{3.11}  and  \ref{3.12}, we obtain  $\Omega^{-1}=\rho_F^{-1}(T_1^\mathfrak N)$ and $${\rm ind}(T_1^\mathfrak N-\lambda)=0, \quad\lambda\in \Omega_1^{-1}.$$ By Lemma~\ref{5.5} iii), ${\rm dimker}(T_1^\mathfrak N-\lambda)=0$. Hence, $\Omega_1^{-1}\subset\rho(T_1^\mathfrak N)$.
Since $\sigma(T^{-1})\subset\sigma((T^{-1})^\mathfrak N)$, $\Omega_1\subseteq \rho(T^{-1})$. Since  $\sigma(T^{-1}|_\mathfrak N)^\land\bigcap\rho_{\rm F}(T^{-1}|_\mathfrak N)$ contains just one connected component $\Omega=\Omega_0\bigcap\rho(T^{-1})$, $\Omega_1\subseteq \rho(T^{-1}|_\mathfrak N)$. By Lemma~\ref{5.2}, $\Omega_1\subseteq \rho((T^{-1})^\mathfrak N)$.
Therefore, for each $\lambda\in\Omega_1^{-1}$,
\[(T^{-1})^\mathfrak N(T_1^\mathfrak N-\lambda)=I-\lambda(T^{-1})^\mathfrak N=\lambda\Big(\frac{1}{\lambda}- (T^{-1})^\mathfrak N\Big).\]
This entails that $(T^{-1})^\mathfrak N(T_1^\mathfrak N-\lambda)$ is invertible, which contradicts to that $(T^{-1})^\mathfrak N\in B_1(\Omega)$ so that $(T^{-1})^\mathfrak N$ is not invertible.
\end{proof}
\section{  Properties A and C}
 For the study related  the  single  valued  extension of the  resolvent $(T-\lambda)^{-1}$ of an operator $T\in\mathfrak B(X)$, N. Dunford \cite{Dun} introduced the following properties (A) and (C).
 \begin{definition}\label{7.1}  Let $T\in \mathcal{B}(X)$.

i)\;$T\in \mathcal{B}(X)$ is said to have Property ($A$) if for all $x\in X$, $(T-\lambda)^{-1}x$ has a largest single valued analytic extension $\tilde{x}(\lambda)$. The domain of $\tilde{x}(\lambda)$, denoted by $\rho_T(x)$, is called the local resolvent of $T$ at $x$; and the local spectrum of $T$ at $x$ is defined by $\sigma_T(x)=\mathbb C\setminus \rho_T(x)$.

ii)\; We say that $T$ has Dunford's Property $C$ if $T$ has Property ($A$) and for every closed subset $F\subseteq \mathbb C$,
\[
M(T,F)\equiv \{x:\,\sigma_T(x)\subseteq F\}
\]
is an invariant subspace of $T$.
\end{definition}

The family of operators admitting Dunford's Property ($C$) is a large class of $\mathcal{B}(X)$. For example, it contains the classes of  generalized spectral operators and of decomposable operators ; and of normal, seminormal and hypo-normal operators (an operator $T$ is called hypo-normal if $T^*T-TT^*\geq 0$), and of $N$-class operators (an operator $T$ is called $N-$class operator if $\|T^N\xi\|^2\leq \|T^{2N}\xi\|\cdot \|\xi\|$ for all $\xi \in \mathcal{H}$) \cite[Theorem 2]{li} in a separable Hilbert space.
The following property is due to Dunford \cite{Dun}.
\begin{proposition}\label{7.2}
Suppose that $T\in \mathcal{B}(X)$ has Property ($A$). Then
\begin{enumerate}
\item $\sigma_T(\alpha x+\beta y)\subseteq \sigma_T(x)\bigcup\sigma_T(y)$ for $x,y\in X$ and $\alpha,\beta\in \mathbb C$;
\item $\sigma_T(x)=\emptyset$ if and only if $x=0$;
\item $\sigma_T(Ax)\subseteq \sigma_T(x)$, where $A\in \mathcal{B}(X)$ with $AT=TA$ and $x\in X$.
\end{enumerate}
\end{proposition}

R.C. Sine \cite{Si} showed the following result.
\begin{proposition}\label{7.3}
Suppose that $T\in \mathcal{B}(X)$ has Property ($A$). Then $$\sigma(T)=\bigcup_{x\in X}\sigma_T(x).$$
\end{proposition}
Chengzu Zou \cite{Zou} proved the following theorem.
\begin{theorem}\label{7.4}
A sufficient and necessary condition for $T\in \mathcal{B}(X)$ to have Property ($A$) is that there exist a  sequence $\{x_n\}\subset X$ of non-zero vectors, $\rho>0$  and $\lambda_0\in\mathbb C$ such that

i) $\|x_n\|\leq\rho^n,\;\;n=0,1,\cdots;$

ii) $Tx_n=x_{n-1}+\lambda_0x_n$, where $x_{-1}=0$.
\end{theorem}

As its application of  Theorem \ref{7.4} and \cite[Proposition 2.21]{LN}, we obtain the following result.
\begin{lemma}\label{7.5}
Suppose that $T\in\mathfrak B(X)$ is transitive invertible,  $T^{-1}$ has Property ($C$), and $M$ is an invariant subspace of $T^{-1}$. Then the restriction $T^{-1}|_M$ admits Property ($A$). Furthermore, $T^{-1}|_M$ has Dunford's property ($C$).
\end{lemma}

\begin{lemma}\label{7.7}
Suppose that $T\in\mathfrak B(X)$ is transitive invertible,  $T^{-1}$ has Property ($C$), and $M$ is an invariant subspace of $T^{-1}$. If $\sigma_{T^{-1}}(x)=\sigma(T^{-1})$ for every $x\in X$, then
\begin{equation}\label{7.8} \sigma(T^{-1})\subset\sigma(T^{-1}|_M)\subset\sigma(T^{-1})^\land.
\end{equation}
\end{lemma}
\begin{proof}
Since $T$ is transitive, $\sigma_p(T^{-1}|_M)=\emptyset$. It follows from Lemma \ref{7.5} that the restriction $T^{-1}|_M$ admits Property ($A$) and ($C$).
We first claim
\begin{equation}\label{7.9} \sigma_{T^{-1}}(x)\subset\sigma_{T^{-1}|_M}(x),\;{\rm for\;all\;}x\in M,
\end{equation}
which is equivalent to  $\rho_{T^{-1}|_M}(x)\subset\rho_{T^{-1}}(x).$

Since  $M\in{\rm Lat}(T^{-1})$, $(T^{-1}|_M-\lambda)^{-1}x\in M.$ Therefore,
\[(T^{-1}-\lambda)(T^{-1}|_M-\lambda)^{-1}x=(T^{-1}|_M-\lambda)(T^{-1}|_M-\lambda)^{-1}x=x.\]
Consequently, $\rho_{T^{-1}|_M}(x)\subset\rho_{T^{-1}}(x).$

Next, we claim $\sigma({T^{-1}|_M})\supset\sigma({T^{-1}}).$ Indeed,  since $\sigma_{T^{-1}}(x)=\sigma(T^{-1})$ for every $x\in X$, by Proposition~\ref{7.3},
\[\sigma(T^{-1})=\bigcup_{x\in M}\sigma_{T^{-1}}(x)\subset\bigcup_{x\in M}\sigma_{T^{-1}|_M}(x)=\sigma(T^{-1}|_M). \]
The second inclusion $\sigma(T^{-1}|_M)\subset\sigma(T^{-1})^\land$ in (\ref{7.8}) follows from Lemma \ref{2.8}.
\end{proof}
\section{Proof of Theorem 1.6}

\begin{theorem}\label{7.10}
 Suppose that $T\in \mathfrak B(X)$ is  invertible such  that $T^{-1}$ admits Dunford's Property ($C$),  and there exists a connected component $\Omega$  of $int\sigma(T^{-1})^\land$ off the origin such that $\Omega\cap\rho_F(T^{-1})\neq \emptyset$. If $T^{-1}$ is intransitive, then $T$ is again intransitive.
\end{theorem}
\begin{proof} We will use Theorem \ref{5.8} to show this result.

Suppose, to the contrary, that $T$ is transitive. Since $T^{-1}$ is intransitive, there is $0\neq x_0\in X$ such that $\mathfrak N=[T^{-k}x_0]_{k\geq1}\neq X$.

In the case that there is $0\neq x\in \mathfrak N$ such that $\sigma_{T^{-1}}(x)\subsetneq \sigma(T^{-1})$,   write $F=\sigma_{T^{-1}}(x)$. Since   $T^{-1}$ admits Property ($C$), and since Property ($C$) implies Property ($A$) (Definition \ref{7.1}),  $F$ is nonempty closed in $\mathbb C$.  Let
\[
M(T^{-1},F)=\{z\in X:\,\sigma_{T^{-1}}(z)\subseteq F\}.
\]
Again  by Definition \ref{7.1}, $M(T^{-1},F)$ is a closed subspace such that $$M(T^{-1},F)\neq X.$$ By Proposition \ref{7.2} (3),
\[\sigma_{T^{-1}}(T^kx)\subset\sigma_{T^{-1}}(Tx)\subset M(T^{-1},F)\subsetneq\sigma(T^{-1}). \]
 By Proposition \ref{7.2} (1), the closed subspace $[T^kx]_{k\geq1}\subset M(T^{-1},F)$.
 We assert that  $M(T^{-1},F)$ is a nontrivial invariant subspace of $T$. Otherwise, due to Sine's theorem  \cite{Si} (i.e., Proposition \ref{7.3}),
 \[M(T^{-1},F)\supset\bigcup_{z\in [T^kx]_{k\geq1}}\sigma_{T^{-1}}(x).\]
 Since $T$ is transitive, $[T^kx]_{k\geq1}=X$. Therefore, $$\sigma_{T^{-1}}=\bigcup_{z\in [T^kx]_{k\geq1}}\sigma_{T^{-1}}(z)\subset M(T^{-1},F). $$ This is a contradiction!

 Thus, we can assume that $\sigma_{T^{-1}}(x)= \sigma(T^{-1})$ for $x\in X$. Since $T^{-1}$ is intransitive, there exists $0\neq x_0\in X$ such that $\sigma_{T^{-1}}(x_0)=\sigma(T^{-1})$
 and $\mathfrak N=[T^{-k}x_0]_{k\geq1}\neq X$. Clearly, $\mathfrak N\in{\rm Lat}T^{-1}$. Let $U=T^{-1}|_\mathfrak N$. By the assumption that $T$ is transitive,  $\sigma_p(U)=\emptyset$. By  Theorem \ref{7.5}, $U$ admits Property ($A$). Consequently, for each $x\in\mathfrak N$, $\sigma_U(x)\neq\emptyset.$

By Lemma~\ref{7.7},
 \begin{equation}\label{7.10} \sigma_{T^{-1}}(x)\subset\sigma_U(x),\;\;{\rm for\; all}\; x\in\mathfrak N \end{equation}
 and
 \begin{equation}\label{7.12}
\sigma(T^{-1})^\land=\sigma(U)^\land.
\end{equation}

Lemma~\ref{2.8}, ~(\ref{7.12}) and the assumption that there exists a connected component $\Omega$  of ${\rm int}\sigma(T^{-1})^\land$ off the origin such that $\Omega\cap\rho_F(T^{-1})\neq \emptyset$  deduce that
  $\sigma(T^{-1}|_\mathfrak N)^\land\cap \rho_F(T^{-1}|_\mathfrak N)$ contains  two connected components.
    By  Theorem \ref{4.14}, $T$ is  intransitive. This is a contradiction to  that $T$ is  transitive.
\end{proof}
\section{Proof of Theorem 1.7}

Before proving Theorem \ref{1.8}, we first recall the definition of strictly cyclic vectors of an operator $T$.
\begin{definition}\label{14.1}
Let $T\in \mathfrak B(X)$ and let $\mathfrak A(T)$ be the closed subalgebra of $\mathfrak B(X)$ generated by $T$ and $I$. A vector $\xi\in$ is called a strictly cyclic vector of $T$ if $\{S\xi:\, S\in\mathfrak A(T)\}=X$. An invariant subspace $Y\subseteq X$ of $T$ is called a strictly cyclic invariant subspace if $T|_Y\in \mathfrak B(Y)$ has a strictly cyclic vector.
\end{definition}

The following lemma is due to B. Barnes~ \cite[Corollary 3]{Bar}.
\begin{lemma}\label{14.2}
Suppose that $X$ is a complex separable Banach space,  and that $T\in \mathfrak B(X)$  admits a strictly cyclic vector. Then $\sigma_p(T^*)=\sigma(T^*)$.
\end{lemma}
For more information related to strictly cyclic vectors, we refer the reader to D.A. Herrero \cite{Her3}.

{The following lemma is a consequence of Proposition~\ref{2.5}.
\begin{lemma}\label{14.3'} Let $X$ be a complex Banach space and $T\in \mathcal{B}(X)$ be an invertible operator. If $[{\rm ran}(T^{-1}-\lambda)]\neq X$, then $T$¡¡is intransitive.
\end{lemma}

\begin{lemma}\label{14.3} Let $X$ be a complex Banach space and $T\in \mathcal{B}(X)$ be an invertible transitive operator. If
there exists $0\neq x_0\in X$ such that $$\mathfrak{N}=[T^{-k}x_0]_{k\geq 1}\subsetneq X$$ and $U=T^{-1}|_{\mathfrak{N}}$ is strictly cyclic, then $\sigma(U)\subseteq \sigma((T^{-1})^{\mathfrak{N}^\perp})$.
\end{lemma}
\begin{proof}
Since $U$ is strictly cyclic, $\sigma_p(U^*)=\sigma(U^*)$ (Lemma~\ref{14.2}). Note that $U^*=((T^*)^{-1})^{\mathfrak{N}^\perp}$ and $((T^{-1})^{\mathfrak{N}^\perp})^*=(T^*)^{-1}|_{\mathfrak{N}^\perp}$. Then it suffices to show that \begin{equation}\label{14.3'}\sigma(((T^*)^{-1})^{\mathfrak{N}^\perp})\subseteq \sigma((T^*)^{-1}|_{\mathfrak{N}^\perp}).\end{equation} By Lemma~\ref{14.2}, there exists  $\lambda\in \sigma_p(((T^*)^{-1})^{\mathfrak{N}^\perp})\bigcap \rho((T^*)^{-1}|_{\mathfrak{N}^\perp})$. Consequently, there exists $0\neq z+\mathfrak{N}^\perp\in X^*/\mathfrak{N}^\perp$ such that
\[
\left(((T^*)^{-1})^{\mathfrak{N}^\perp}-\lambda \right)(z+\mathfrak{N}^\perp)=0.
\]
 Therefore, $((T^*)^{-1}-\lambda)z\in \mathfrak{N}^\perp$. Since $\lambda\in \rho((T^*)^{-1}|_{\mathfrak{N}^\perp})$, there exists  $z_1\in \mathfrak{N}^\perp$ such that
\[
((T^*)^{-1}-\lambda)z_1=-((T^*)^{-1}-\lambda)z.
\]
Equivalently,
\[
((T^*)^{-1}-\lambda)(z_1-z_0)=0.
\]
This says that $\ker((T^*)^{-1}-\lambda)\neq \{0\}$. Hence, $[{\rm ran}(T^{-1}-\bar{\lambda})]\neq X$. By Lemma~\ref{14.3'}, $T$ is intransitive, and this is a contradiction.

\end{proof}

}

Now, we are ready to show Theorem \ref{1.7}. Let us restate it as follows.
\begin{theorem}\label{14.4}
 Suppose that $T\in \mathfrak B(X)$ is invertible and  that ${\rm int}\sigma(T^{-1})^\land$ contains  two components $\Omega_0$,  $\Omega_1$ with $0\in\Omega_0$  and $\Omega_1\bigcap\rho(T^{-1})\neq\emptyset$. If $T^{-1}$ has a proper strictly cyclic invariant subspace, then $T$ is intransitive.
\end{theorem}
\begin{proof} Suppose, to the contrary, that $T$ is transitive.
Since $T^{-1}$ has a proper strictly cyclic invariant subspace, there is $0\neq x_0\in X$ such that $$\mathfrak N=[T^{-k}x_0]_{k\geq1}\in\mbox{Lat}T^{-1}$$ and $x_1\equiv T^{-1}x_0$ is a strictly cyclic vector of $U\equiv T^{-1}|_\mathfrak N$. By Lemma~\ref{14.2}, $$\sigma_p(U^*)=\sigma(U^*).$$
By Lemma \ref{14.3}, $\sigma(U)\subset\sigma((T^{-1})^\mathfrak N)$. Applying Theorem \ref{5.8}, we obtain that  that $T$ is intransitive, and this is a contradiction!
\end{proof}
\section{Proof of Theorem 1.8}
A remarkable result of C.J. Read \cite{R4} states that for an infinite dimensional separable complex Banach space $X$, if $X=\ell_1\oplus Y$ for some separable Banach space $Y$ then $\mathfrak B(X)$ admits an operator $T\in\mathfrak B(X)$ without any non-trivial invariant subspace. With the help of this result, we  show the following theorem.

\begin{theorem}
A sufficient and necessary condition for an infinite dimensional $L_1(\Omega,\sum,\mu)$ admitting an operator $T\in\mathfrak B(L_1(\Omega,\sum,\mu))$ without any nontrivial subspace is that the measure space
$(\Omega,\sum,\mu)$ is $\sigma$-finite.
\end{theorem}
\begin{proof}
Sufficiency. Note that if $(\Omega,\sum,\mu)$ is $\sigma$-finite, then $L_1(\Omega,\sum,\mu)$ is separable.  Then by a theorem of Read \cite{R4}, it suffices to show that $\ell_1$ is a
complemented in $X\equiv L_1(\Omega,\sum,\mu)$.
Since $(\Omega,\sum,\mu)$ is $\sigma$-finite, and since $L_1(\Omega,\sum,\mu)$ is infinite dimensional,
 there is a $\sum$-partition  $\{E_n\}$ of $\Omega$ so that $0<\mu(E_n)<\infty$ for all $n\in\N$. Therefore,
\begin{equation}\nonumber
f=\sum_{n=1}^\infty f\cdot\chi_{E_n},\;\forall f\in X,\end{equation}
where $\chi_{E_n}$ is the set indicator function of $E_n$ for all $n\in\N$.
 Let $X_0=[\chi_{E_n}]$, the norm closure of the linear hull span$\{\chi_{E_n}\}_{n=1}^\infty$ in $X$.
Then each $f\in X_0$ can be represented as
\[f=\sum_{n=1}^\infty\alpha_n\chi_{E_n},\]
where $\alpha_n=\int_{E_n}fd\mu,\;\;n=1,2,\cdots$.

Next, we define
\[Sf=\sum_{n=1}^\infty\beta_ne_n,\;f=\sum_{n=1}^\infty\alpha_n\chi_{E_n}\in X_0,\]
where $\{e_n\}$ is the standard unit vector basis of $\ell_1$ and
\[\beta_n=\mu(E_n)\cdot\alpha_n, \;n=1,2,\cdots.\]
Then it is easy to check that  $S$ is a   linear surjective isometry  from $X_0$ to $\ell_1$.

Finally, let $P: X\rightarrow X_0$ defined for $f\in X$ by
\begin{equation}\nonumber
Pf=\sum_{n=1}^\infty \alpha_n\cdot\chi_{E_n},\;\forall f\in X,\end{equation}
where $\alpha_n=\int_{E_n}fd\mu,\;\;n=1,2,\cdots$.
Clearly, $P: X\rightarrow X_0$ is a projection of norm one. This entails that $X_0$ is 1-complemented in $X$.
Consequently, $\ell_1$ is 1-complemented in $X$.

Necessity. Since every bounded operator on a non-separable Banach space always admits nontrivial invariant subspace, and since $L_1(\Omega,\sum,\mu)$ is separable if and only if  $(\Omega,\sum,\mu)$ is $\sigma$-finite, the necessity follows.
\end{proof}

\section{Proof of Theorem 1.9}
Based on another beautiful result of Read \cite{R5} that  an infinite dimensional separable complex Banach space $X$  of the form $X=c_0\oplus Y$ for some separable $Y$ always admits an operator $T\in\mathfrak B(X)$ without any non-trivial invariant subspace, and the Sobczyk theorem \cite{So} that $c_0$ is $2$-complemented in every separable Banach space geometrically containing  it, in this section, we show the following theorem.

\begin{theorem} Let $K$ be a complete metric space such that the space $C(K)$ of all bounded continuous complex valued functions is infinite dimensional. Then
 a sufficient and necessary condition for $C(K)$ admitting an operator $T\in\mathfrak B(C(K))$ without any nontrivial subspace is that the metric space $K$
 is compact.
\end{theorem}
\begin{proof}
Sufficiency.
Since $K$
 is a compact metric space,  $C(K)$ is separable. Since $C(K)$ is infinite dimensional, the cardinality of $K$ is infinite. Let $K^\prime$ be the set of all accumulation points of $K$.  If $K$ is uncountable, then $C(K)$ is universal for the class of separable Banach spaces over $\mathbb C$ (See, for instance, \cite[p. 229, Corollary]{Ho}), that it geometrically contains every separable Banach space, hence, $c_0$. Therefore, by the Sobczyk theorem \cite{So}, $c_0$ is $2$-complemented in $C(K)$. Thus, the sufficiency follows from Read's theorem \cite{R5}. If $K$ is countable, then   $K=(K\setminus K^\prime)\bigcup K^\prime$ and $K\setminus K^\prime$ is a countable infinite subset of $K$. Fix any convergent subsequence $\{x_n\}_{n=1}^\infty\subset K\setminus K^\prime $, say, $x_n\rightarrow x_0$. Let \[C_0(K)=\big\{f: K\rightarrow\mathbb F, f(x_n)\rightarrow0, {\rm as\;}n\rightarrow\infty;\;f(x)=0,\;x\in K\setminus\{x_n\}_{n=1}^\infty\big\}.\]
 Then it is easy to observe that $C_0(K)$ is a closed subspace of $C(K)$ isometric to $c_0$. Therefore, again by the Sobczyk theorem \cite{So}, $c_0$ is $2$-complemented in $C(K)$.
 Therefore, the sufficiency follows from the previously mentioned theorem of Read \cite{R5}.

Necessity. Note that every bounded linear operator on a nonseparable Banach space admits a nontrivial invariant subspace. Since $K$ is a complete metric space,  and since that there is $T\in\mathfrak B(C(K))$ without any nontrivial subspace, $C(K)$ must be separable and infinite dimensional. Therefore, the cardinality of $K$ is infinite. If $K$ is not compact, then it is not difficult to observe that $\ell_\infty$ can be linearly isometrically embedded into $C(K).$ This contradicts to that $C(K)$ is separable.

\end{proof}


\bibliographystyle{amsalpha}

\end{document}